\documentclass[leqno]{report}

\usepackage{amsmath,amsthm,amsfonts,enumerate,verbatim,graphicx,wrapfig,epsfig,longtable,listings}
\usepackage{rac,dbl12,caption}
\usepackage[scaled=.8]{DejaVuSansMono}
\usepackage[all,arc]{xy}

\usepackage{mathrsfs,amssymb}

\usepackage[T2A,T1]{fontenc}
\newcommand{\Sha}{\mbox{\usefont{T2A}{\rmdefault}{m}{n}\CYRSH}}

\newcommand{\defof}[1]{\textbf{#1}}

\def\A{\ensuremath{\mathbb{A}}}
\def\C{\ensuremath{\mathbb{C}}}
\def\N{\ensuremath{\mathbb{N}}}
\def\P{\ensuremath{\mathbb{P}}}
\def\R{\ensuremath{\mathbb{R}}}
\def\Z{\ensuremath{\mathbb{Z}}}

\def\SS{\mathfrak{S}}
\def\TT{\mathfrak{T}}
\def\AA{\mathfrak{A}}
\def\II{\mathfrak{I}}

\DeclareMathOperator\codim{codim}
\DeclareMathOperator\ec{ec}
\DeclareMathOperator\rk{rk}
\DeclareMathOperator\mspan{span}

\newcommand{\inj}{\hookrightarrow}

\def\Mat{\ensuremath{\mathrm{Mat}}}

\theoremstyle{plain}
\newtheorem{thm}{Theorem}
\newtheorem{prop}[thm]{Proposition}
\newtheorem{lem}[thm]{Lemma}
\newtheorem{cor}[thm]{Corollary}

\theoremstyle{definition}
\newtheorem{cons}[thm]{Construction}
\newtheorem{defn}[thm]{Definition}
\newtheorem{defns}[thm]{Definitions}
\newtheorem{exmp}[thm]{Example}

\newtheorem{cex}[thm]{Counterexample}

\numberwithin{thm}{chapter}
\title{Geometric Shifts and Positroid Varieties}
\author{Nicolas Ford}

\begin{document}

\bibliographystyle{plain}

\titlepage{GEOMETRIC SHIFTS AND POSITROID VARIETIES}{Nicolas Ford}{Doctor of Philosophy}
{Mathematics} {2014}
{ Associate Professor David E. Speyer, Chair \\
  Associate Professor Henriette Elvang\\
  Professor Sergey Fomin\\
  Professor Thomas Lam \\
  Professor Karen E. Smith }

\initializefrontsections
\setcounter{page}{1}
\tableofcontents
\listoffigures

\startthechapters 

\chapter{Introduction}
Consider a point $x$ on the Grassmannian $G(k,n)$ of $k$-planes in $\C^n$. The \defof{matroid} of $x$ is defined to be the set of Pl\"ucker coordinates that are nonzero at $x$, and a \defof{matroid variety} is the closure of the set of points on $G(k,n)$ with a particular matroid. Many enumerative problems on the Grassmannian can be described in terms of matroid varieties; the Schubert varieties that form the usual basis for the cohomology ring of $G(k,n)$ are an especially well-behaved special case.

So if there were a good way to take a matroid and produce the cohomology class of the corresponding matroid variety, we would have a systematic way of solving any enumerative problem on the Grassmannian that can be described purely in terms of the vanishing or nonvanishing of Pl\"ucker coordinates. Unfortunately, in a sense that will be made more precise later, matroid varieties are very poorly behaved, and there is essentially no way to solve the problem just described in a reasonable and efficient manner: there is a strictly easier problem that is known to be NP-hard.

One might expect to have more success with a slightly more modest goal: find a well-behaved class of matroids for which the class of the corresponding matroid variety can be described nicely. There is a class of matroids called \defof{positroids} which are a natural candidate: the corresponding matroid varieties, called \defof{positroid varieties} are geometrically much better behaved, and indeed some work toward the goal of describing their cohomology classes has already been done. On the one hand, in \cite{klspos} there is a description of the cohomology class of a positroid variety in terms of a symmetric function called an ``affine Stanley symmetric function.'' When the cohomology ring of the Grassmannian is expressed as a quotient of the ring of symmetric functions, the affine Stanley function maps to the class of the corresponding positroid variety.

It would be nice, though, if there were a way to do the computation directly in the cohomology ring: the symmetric function description involves a mysterious change of basis that introduces several minus signs that all disappear after taking the quotient. What would be preferable would be a more combinatorial description that computes, from the positroid, the coefficient of each Schubert class in its cohomology class. Indeed, for an even more restricted class of matroid varieties called \defof{interval rank varieties}, there is a description that does exactly that. Given an interval rank variety, there is a sequence of degenerations in the Grassmannian that take it to a union of Schubert varieties, and there is a combinatorial procedure for keeping track of what happens along the way.

This thesis is in two parts. The first is an account of an attempt to extend the procedure that worked for interval rank varieties to positroids. We will see that, with a lot of arm-twisting, it can be made to work when $k\le 3$ (Theorem \ref{thm:rank3pos}), but when $k\ge 4$ it fails (Counterexample \ref{cex:squarepos}). In the second part, we have more success at an even more modest goal: it describes a new way to estimate just the codimension of a matroid variety purely in terms of the combinatorics of the matroid itself. The resulting number is not always the actual codimension, but we prove that in the case of positroids, it is. This work, which appears in Section 6, also appears in a paper \cite{fordecodim}, which has been submitted to the \textit{Journal of Algebraic Combinatorics}.

We write $[n]$ for the set $\{1,2,\ldots,n\}$, and for any set $S$ we write $\mathcal{P}(S)$ for the power set of $S$ and $\binom Sk$ for the set of all $k$-element subsets of $S$. $G(k,n)$ will always stand for the Grassmannian of $k$-planes in $\C^n$. For $S\in\binom{[n]}k$, we write $p_S$ for corresponding Pl\"ucker coordinate on $G(k,n)$; that is, thinking of elements of $G(k,n)$ as being represented by $k\times n$ matrices, $p_S$ is the determinant of the minor whose columns correspond to the elements of $S$. All varieties are over $\C$.

Fix a complete flag $0=V_0\subset V_1\subset\cdots\subset V_n=\C^n$. One may define a \defof{Schubert variety} from a partition $\lambda$ consisting of parts $n-k\ge\lambda_1\ge\cdots\ge\lambda_k\ge 0$. We define $\Omega_\lambda$ to be the set of all $W\in G(k,n)$ such that, for each $i$, \[\dim(W\cap V_{n+k+i-\lambda_i})\ge i.\] The corresponding cohomology classes are called \defof{Schubert classes}, written $\sigma_\lambda=[\Omega_\lambda]$. The Schubert classes corresponding to partitions that fit in a $k\times (n-k)$ box --- that is, partitions with at most $k$ parts each of which is $\le n-k$ --- form a basis for the cohomology ring of the Grassmannian. When a basis is chosen for $\C^n$, we often take the flag to be the one corresponding to that basis. In this case, taking the Schubert variety corresponding to the flag given by taking the basis in reverse order produces what we'll call an \defof{opposite Schubert variety}.

We write $\Lambda_k$ for the ring of symmetric polynomials in $k$ variables. The well-known \defof{Schur polynomials} corresponding to partitions with at most $k$ parts give a basis for $\Lambda_k$. There is a surjective map $\Lambda_k\to H^*(G(k,n))$ which takes each Schur polynomial $s_\lambda$ to $\sigma_\lambda$; its kernel is spanned by the $s_\lambda$ corresponding to partitions $\lambda$ with a part larger than $n-k$. Proofs of every claim in this and the preceding paragraph can be found in \cite[Ch. 3]{manivel}.

We will at some points have use for the $GL_k$-equivariant cohomology ring $H^*_{GL_k}(\Mat_{k\times n})$, where $\Mat_{k\times n}$ is the variety consisting of all $k\times n$ matrices. This ring is isomorphic to $\Lambda_k$ itself. Restricting to the subvariety $\Mat_{k\times n}^\circ$ of full-rank $k\times n$ matrices causes the action of $GL_k$ to be free, making the corresponding equivariant cohomology ring isomorphic to the cohomology ring of $\Mat_{k\times n}^\circ/GL_k\cong G(k,n)$. The restriction map \[\Lambda_k=H^*_{GL_k}(\Mat_{k\times n})\to H^*_{GL_k}(\Mat_{k\times n}^\circ)\cong H^*(G(k,n))\] is the same as the map described in the previous paragraph.

To see this, we first note that there is a $GL_k$-equivariant contraction of $\Mat_{k\times n}$, which allows us to just consider $H^*_{GL_k}(\mathit{pt})$. This, by Section 1 of \cite{andersonequivariant}, is the copy of $\Lambda_k$ sitting inside $\Z[x_1,\ldots,x_k]=H^*_{(\C^*)^k}(\mathit{pt})$, so it's enough to compute the $(\C^*)^k$-equivariant cohomology class of the matrix Schubert varieties. This is done in Section 3 of \cite{andersonequivariant}.

\chapter{Matroids and Matroid Varieties}
\section{Matroids}

We are going to be investigating subvarieties of Grassmannians defined by combinatorial objects called \defof{matroids}. There are many equivalent definitions of matroids, all useful in different contexts, and we are only going to mention two of them here. There is a lot of literature on the combinatorial theory of matroids. A good place to start might be \cite{whitemat}.

From our perspective, the purpose of a matroid is, given a collection of vectors in a vector space, to combinatorially capture the information about the linear relations among the vectors. We will consider two equivalent ways to do this. Details of these and other axiomatizations of matroids can be found in \cite[pp. 298--312]{whitemat}.

\begin{defn}
\label{def:basismat}
A matroid may be specified in terms of its \defof{bases}. According to this definition, a matroid $M$ is a finite set $E$ together with a collection of subsets $\mathscr{B}\subseteq\mathcal{P}(E)$. (An element of $\mathscr{B}$ is called a basis.) We require:
\begin{itemize}
\item $\mathscr{B}$ is not empty.
\item No element of $\mathscr{B}$ contains another.
\item For $B,B'\in\mathscr{B}$ and $x\in B$, there is some $y\in B'$ so that $B-\{x\}\cup\{y\}\in\mathscr{B}$.
\end{itemize}
Note that this is enough to force all bases to have the same number of elements.
\end{defn}

\begin{defn}
\label{def:matroidvector}
Suppose we have a finite-dimensional vector space $V$, a finite set $E$, and a function $e:E\to V$ whose image spans $V$. We can put a matroid structure on $E$ by taking $\mathscr{B}$ to be the collection of all subsets of $E$ which map \emph{injectively} to a basis of $V$. (The reason for this funny definition is that we'd like to be able to take the same element of $V$ more than once; otherwise $E$ could just be a subset of $V$. We will hardly ever be very careful about the difference between an element $x\in E$ and its image $e(x)\in V$.) It's an easy exercise in linear algebra to show that this definition satisfies the axioms above. Matroids which arise in this way are called \defof{realizable}.
\end{defn}

It will sometimes be convenient to be able to draw a picture of a realizable matroid by drawing the images of the elements of $E$ in the projective space $\P(V)$ and the linear subspaces they lie on. We will call these \defof{projective models}. For example, Figure \ref{fig:vmat} shows a projective model of the rank-3 matroid on $\{1,2,3,4,5\}$ in which $\{1,2,3\}$ and $\{3,4,5\}$ are the only three-element sets that are not bases.
\begin{figure}[t]
\begin{center}\includegraphics[width=6cm]{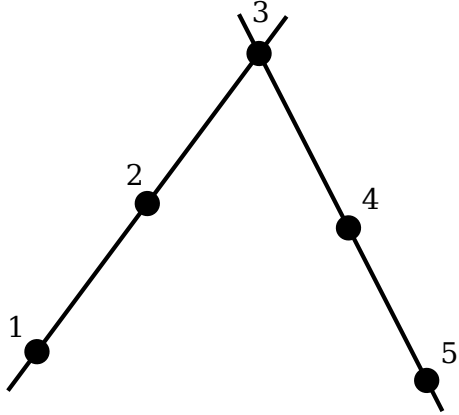}\end{center}
\caption{A projective model of the ``V'' matroid.}
\label{fig:vmat}
\end{figure}

The following terminology will be useful to have as we talk about matroids, so we present it here. Most (but not all) of the terminology mirrors the corresponding terminology from linear algebra in the realizable case, which should make those terms easier to remember. Throughout this discussion, unless otherwise specified, $M$ is a matroid on the set $E$.

\begin{defns}
\label{def:matroiddefs}
~\begin{enumerate}
\item A set which is contained in a basis is called \defof{independent}. Any other set is \defof{dependent}.
\item For $F\subseteq E$, the \defof{rank} of $F$, written $\rk F$, is the size of the largest independent set contained in $F$. Note that $\rk E$ is the same as the size of any basis. We will sometimes write $\rk M=\rk E$.
\item For a set $F$ and an element $x\in E$, we say that $x$ is in the \defof{closure} of $F$, written $x\in\overline F$, if $\rk(F\cup\{x\})=\rk F$. Note that, as the name suggests, closure is idempotent and inclusion-preserving. Sets which are their own closures are called \defof{flats}.

In the realizable case, the flats are the intersections of subspaces of $V$ with $E$.
\item A set $F$ which contains a basis is called a \defof{spanning set}. Equivalently, $F$ spans if $\rk F=\rk E$, or if $\overline F=E$.
\item If $\rk(\{x\})=0$, we say $x$ is a \defof{loop}. Equivalently, $x$ is not in any basis, or $x\in\overline\varnothing$, or $x$ is in every flat.

In the realizable case, loops are elements of $E$ which map to the zero vector in $V$.
\item If $\rk(E-\{x\})=\rk E-1$, we say $x$ is a \defof{coloop}. Equivalently, $x$ is in every basis.
\item If $\rk(\{x,y\})=1$, we say that $x$ and $y$ are \defof{parallel}. Equivalently, any flat which contains one of $x$ or $y$ also contains the other.
\item Given a set $F\subseteq E$, we can put a matroid structure on $F$ by saying $B\subseteq F$ is a basis of $F$ if it is maximal among independent sets contained in $F$. Note that this preserves the independence and dependence of subsets of $F$, and therefore also the ranks of subsets of $F$. This operation is called \defof{restricting} to $F$ or \defof{deleting} $E-F$, and we call the resulting matroid $M|F$ or $M-(E-F)$.

In the realizable case, this corresponds to restricting the function $e$ to $F$ and replacing $V$ with the span of the image of $F$.
\item There is a second way to put a matroid structure on $E-F$ called the \defof{contraction}, and written $M/F$. The rank function on the contraction is $\rk_{M/F}(A)=\rk(A\cup F)-\rk(F)$. 
\end{enumerate}
\end{defns}

It will also be convenient to note that matroids can be defined just by listing the axioms that have to be satisfied by the rank function defined above:

\begin{defn}
\label{def:rankmat}
A matroid may be specified in terms of the ranks of all its subsets. According to this definition, a matroid is a finite set $E$ together with a function $\rk:\mathcal{P}(E)\to\N$ satisfying:
\begin{itemize}
\item $\rk\varnothing=0$.
\item $\rk(F\cup\{x\})$ is either $\rk F$ or $\rk F+1$.
\item If $\rk F=\rk(F\cup\{x\})=\rk(F\cup\{y\})$, then $\rk(F\cup\{x,y\})=\rk F$.
\end{itemize}
Note that this is enough to force the useful inequality $\rk A+\rk B\ge\rk(A\cup B)+\rk(A\cap B)$.
\end{defn}

Given the same data we had to define a realizable matroid before --- a set $E$ with a function $e$ to a vector space $V$ --- we can get a rank function on $E$ by setting $\rk(F)=\dim(\mspan(e(F)))$.

We have already mentioned how to turn a collection of bases into a rank function. To go the other way, we can say $B$ is a basis if it is minimal among sets of maximal rank. One can check that these two correspondences make the two definitions given here equivalent. We will not distinguish between them as we go forward.

One more definition will be useful in later parts of this paper:

\begin{defn}
\label{def:matroidgen}
A \defof{pseudo-rank function} on a set $E$ is a function $r:\mathcal{P}(E)\to\N$ satisfying the first two criteria in Definition \ref{def:rankmat}.

Take some collection $\mathscr{S}$ of subsets of $E$ and some function $r:\mathscr{S}\to\N$. The \defof{pseudo-rank function generated by $r$} is the pointwise-largest pseudo-rank function on $\mathcal{P}(E)$ which agrees with $r$ on the sets in $\mathscr{S}$, if such a pseudo-rank function exists. (Note that, since the pointwise max of two pseudo-rank functions is still a pseudo-rank function, this is well-defined if it exists.) If this pseudo-rank function happens to be the rank function of a matroid $M$, we will sometimes say that $M$ is ``generated by the rank conditions given by $r$'' or that $M$ is the matroid ``defined by imposing the rank conditions given by $r$.''
\end{defn}

For example, the matroid on $[4]$ defined by imposing the conditions that $\{1,2\}$ has rank 1 and $\{1,2,3,4\}$ has rank 3 will give all one-element sets rank 1, all two-element sets other than $\{1,2\}$ rank 2, and both three-element sets other than $\{1,2,3\}$ and $\{1,2,4\}$ rank 3. Its bases are $\{1,3,4\}$ and $\{2,3,4\}$.

\section{Matroid Varieties}

As mentioned above, the main objects of study in this paper are certain subvarieties of Grassmannians which can be described in terms of matroids.

\begin{cons}
\label{cons:pointmat}
Consider the Grassmannian $G(k,n)$, which we'll think of as the set of $k\times n$ matrices of full rank modulo the obvious action of $GL_k$. When one builds the Grassmannian in this way, one ordinarily considers the $k$ rows of the matrix as elements of $\C^n$, and the action of $GL_k$ corresponds to automorphisms of the span of those elements, so that we are left with a variety that parametrizes the $k$-planes in $\C^n$.

We will think about our matrices the other way. Given a $k\times n$ matrix of full rank, consider the function $e:[n]\to\C^k$ which takes $i$ to the $i$'th column of our matrix. We can then use Definition \ref{def:matroidvector} to put a matroid structure on $[n]$. Since the action of $GL_k$ clearly doesn't change which matroid we get, we have assigned a matroid in a consistent way to every point of the Grassmannian. The Pl\"ucker coordinate $p_S$ corresponding to some $S\in\binom{[n]}k$ is given by the determinant of the submatrix defined by taking the columns in $S$. So $p_S$ vanishes precisely when these $k$ columns fail to span $\C^k$, that is, precisely when $S$ fails to be a basis of our matroid.

Given a matroid $M$ of rank $k$ on $[n]$, the \defof{open matroid variety} $X^\circ(M)$ is the subset of $G(k,n)$ consisting of all points whose matroid is $M$. This is a locally closed subvariety of $G(k,n)$: it is defined by taking all Pl\"ucker coordinates corresponding to bases of $M$ to be nonzero and all the other Pl\"ucker coordinates to be zero. The closure of $X^\circ(M)$ is called the \defof{matroid variety} $X(M)$. We can define a matroid variety inside $\Mat_{k\times n}$ in the same way. We write $V^\circ(M)$ for the open matroid variety in $\Mat_{k\times n}$ and $V(M)$ for its closure. The open matroid variety in $\Mat_{k\times n}$ doesn't intersect the subvariety of matrices of less than full rank, but its closure will.
\end{cons}

The reader who is familiar with the definition of Schubert varieties may be tempted to ignore part of the definition above and take $X(M)$ to be the subvariety of $G(k,n)$ defined by setting all the Pl\"ucker coordinates corresponding to nonbases of $M$ to zero. Sadly, this is not the same as the definition given above.

\begin{cex}
\label{cex:pluckergen}
Consider the rank-3 matroid $A$ on $[7]$ generated by the conditions that $\{1,2,7\}$, $\{3,4,7\}$, and $\{5,6,7\}$ have rank 2. The variety $X(A)$ is not cut out by the ideal $(p_{127},p_{347},p_{567})$. That ideal cuts out two components: $X(A)$ and the variety of the matroid in which 7 is a loop. The ideal of $X(A)$ is actually \[(p_{127},p_{347},p_{567},p_{124}p_{356}-p_{123}p_{456}).\]
\end{cex}

Matroid varieties can be used to encode any number of enumerative geometry problems. For example, Schubert varieties are matroid varieties, as are intersections of Schubert varieties and opposite Schubert varieties. Of course, there are many more matroid varieties than just these, so coming up with a way to find the cohomology class of a matroid variety would enable one to solve a much larger set of combinatorial problems about linear arrangements of points. Since the multiplication rule for Schubert classes is well-known, it would be enough to come up with an algorithm that takes in a matroid and outputs its class as a linear combination of Schubert classes.

In general, matroid varieties are under no obligation to be geometrically well-behaved. They don't have to be irreducible, equidimensional, normal, or even generically reduced. It is too much to hope that we might find an algorithm that can efficiently produce the class of an arbitrary matroid variety: even the problem of determining whether a given matroid variety is empty or not is NP-hard \cite{stretch}. To make progress on this question, it will be necessary to be more modest in our goals and only look for the classes of certain well-behaved matroid varieties. It is to that task that the rest of this paper is dedicated.

\section{Operations on Matroids and Matroid Varieties}

We first establish some results which describe the effects of some simple operations on the cohomology class of a matroid. We will use these later in the paper to help us describe the classes of a few matroid varieties.

\begin{defns}
\label{def:directsum}
~\begin{enumerate}
\item Let $M$ be a matroid on $E$ and $N$ be a matroid on $F$. The \defof{direct sum} of $M$ and $N$ is the matroid $M\oplus N$ on $E\sqcup F$ defined by \[\rk_{M\oplus N}(S)=\rk_M(S\cap E)+\rk_N(S\cap F).\]
\item If $M$ is a matroid, the \defof{loop extension} of $M$ is the matroid $M\oplus x_0$ formed by taking the direct sum of $M$ with the unique matroid of rank 0 on the one-element set $\{x\}$, so that the new element $x$ is a loop.
\item The \defof{coloop extension} of $M$ is the matroid $M\oplus x_1$ formed by taking the direct sum of $M$ with the unique matroid of rank 1 on $\{x\}$, so that $x$ is a coloop.
\end{enumerate}
\end{defns}

It's straightforward to compute the cohomology class of $X(M\oplus N)$ given the classes of $X(M)$ and $X(N)$.

\begin{prop}
\label{prop:directsum}
Given a matroid $M$ of rank $k$ on $[m]$ and $N$ of rank $l$ on $[n]$, we can think of them both as matroids on $[m+n]=[m]\sqcup[n]$ in the natural way: $M$ puts conditions on the points in $\{1,\ldots,m\}$ and $N$ puts conditions on the points in $\{m+1,\ldots,m+n\}$. Interpreted in this way, $[X(M\oplus N)]=[X(M)][X(N)]$ in $G(k+l,m+n)$.
\end{prop}
\begin{proof}
This result is easier to see in $H^*_{GL_{k+l}}(\Mat_{(k+l)\times(m+n)})$, which is enough to prove the statement for the Grassmannian as well. In fact, $X(M\oplus N)$ is the transverse intersection of $X(M)$ and $X(N)$. The tangent space of $\Mat_{(k+l)\times(m+n)}$ at any point is naturally $\Mat_{(k+l)\times(m+n)}$ itself. At any point of $X(M)\cap X(N)$, the tangent space of $X(M)$ contains the span of the columns in $\{m+1,\ldots,m+n\}$, and the tangent space of $X(N)$ contains the span of the columns in $\{1,\ldots,m\}$, so together they span the entirety of $\Mat_{(k+l)\times(m+n)}$.
\end{proof}

\begin{cor}
\label{cor:loopcoloop}
Let $X(M)\subseteq G(k,n)$ be a matroid variety. Then the class of $X(M\oplus x_0)$ in $H^*(G(k,n+1))$ is $\sigma_{1^k}\cdot[X(M)]$, and the class of $X(M\oplus x_1)$ in $H^*(G(k+1,n+1))$ is $\sigma_{n-k}\cdot[X(M)]$.\qed
\end{cor}

The next definition is a bit more complicated than the ones that precede it. It gives a way to add a new element to a matroid that is similar to the coloop extension except that it doesn't increase the total rank.

\begin{defn}
\label{def:freeext}
Let $M$ be a matroid of rank $k$ on a set $E$. The \defof{free extension} of $M$ by $x$ is the matroid $M+x$ on $E\sqcup\{x\}$ that we get by adding a new, unconstrained element $x$ in a way that does not increase the total rank. Its rank function is defined by\[\rk_{M+x} S=\left\{
\begin{array}{rl}
\rk_M S,&x\notin S\\
\rk_M(S-\{x\})+1,&x\in S\mbox{ and }\rk_M(S-\{x\})<k\\
k,&\rk_M(S-\{x\})=k
\end{array}
\right.\]
\end{defn}

The class of a free extension can also be described in terms of the class of the original matroid. This result will be much easier to state and deal with if we work in $H^*_{GL_k}(\Mat_{k\times n})$ instead of $H^*(G(k,n))$.

\begin{prop}
\label{prop:freeext}
If $M$ is a matroid, the equivariant cohomology classes of $X(M)$ in $H^*_{GL_k}(\Mat_{k\times n})$ and of $X(M+x)$ in $H^*_{GL_k}(\Mat_{k\times (n+1)})$ are represented by the same symmetric function.
\end{prop}
\begin{proof}
The map that takes $X(M)$ to $X(M+x)$ is the pullback in $GL_k$-equivariant cohomology along the projection map $\pi:\Mat_{k\times(n+1)}\to\Mat_{k\times n}$ that kills the last column. But both $\Mat_{k\times (n+1)}$ and $\Mat_{k\times n}$ are contractible, so both of their equivariant cohomology rings are canonically isomorphic to $H^*_{GL_k}(\mathit{pt})=\Lambda_k$, and therefore, after identifying the three cohomology rings, $\pi^*$ is just the identity on $\Lambda_k$.
\end{proof}

\begin{defn}
\label{def:connected}
If $M$ can't be written as a direct sum in a nontrivial way, we say that $M$ is \defof{connected}. If we write $M=\bigoplus_iA_i$ with each $A_i$ connected, then the $A_i$'s are uniquely determined, and we call them the \defof{connected components} of $M$.
\end{defn}

There are two other, equivalent ways to define connectedness \cite[pp. 108-110]{whitemat}:
\begin{itemize}
\item $M$ is connected if there is no proper, nonempty subset $S\subseteq E$ for which $\rk S+\rk(E-S)=\rk E$.
\item A \defof{circuit} of $M$ is a minimal dependent set, that is, a dependent set $C$ for which every proper subset is independent. We can define an equivalence relation on $E$ by saying $x$ is equivalent to $y$ if either $x=y$ or there is a circuit of $M$ containing both $x$ and $y$. The connected components of $M$ are the equivalence classes under this relation.
\end{itemize}

\begin{defn}
\label{def:dual}
Let $M$ be a matroid of rank $k$ on a set $E$, with $\#E=n$. The \defof{dual} of $M$ is the matroid $M^*$ on $E$ whose bases are exactly the complements of bases of $M$.
\end{defn}

The rank of a set $S$ in $M^*$ works out to be $\#S-k+\rk(E\setminus S)$. In particular $M^*$ has rank $n-k$.

\begin{prop}
\label{prop:dual}
Consider the map $\omega:H^*(G(k,n))\to H^*(G(n-k,n))$ defined by taking $\sigma_\lambda$ to $\sigma_{\lambda^{\vee}}$ and extending linearly. Then $[X(M^*)]=\omega([X(M)])$.
\end{prop}
\begin{proof}
In fact, $\omega$ is the map on cohomology induced by the isomorphism $G(n-k,n)\to G(k,n)$ that takes each Pl\"ucker coordinate $p_S$ to $p_{[n]-S}$. Some set $S$ is a basis for $M^*$ if and only if $[n]-S$ is a basis for $M$, so we see that this isomorphism takes $X(M^*)$ to $X(M)$.
\end{proof}

It is important to note that restriction and contraction are dual to each other. That is, $(M|_S)^*=M^*/(E-S)$, and $(M/S)^*=M^*|_{E-S}$.

\chapter{Interval Rank Varieties and the Geometric Shift}
\label{sec:irank}
We mentioned earlier that intersections of Schubert varieties and opposite Schubert varieties are a special case of matroid varieties. These are called \defof{Richardson varieties}. Because Schuberts and opposite Schuberts are transverse, a Richardson variety is a representative of the product of the cohomology classes of the Schuberts used to construct it. So finding an algorithm for expressing the cohomology class of a Richardson variety in terms of Schuberts is the same as finding a way to multiply Schubert classes.

Such an algorithm is called a \defof{Littlewood-Richardson rule}, and, as mentioned in the introduction, there are already many different Littlewood-Richardson rules that can be described by lots of different types of combinatorial gadgets. The rule that we are going to look at in detail here is the one first described by Ravi Vakil in \cite{vakilir} and later in different language and more generality by Allen Knutson in \cite{knutir}. Their approach has a distinct advantage over other Littlewood-Richardson rules in that it can be described purely geometrically. That is, Vakil and Knutson start with a Richardson variety, perform a specific sequence of degenerations, and end up with Schubert varieties at the end. By understanding the varieties that show up in the middle of the sequence of process and how they behave under this operation, one can read off the coefficient of some Schubert class simply by counting how many times the corresponding Schubert variety appears at the end of this process.

What will interest us about this rule is the fact that it does more than just provide a way to multiply Schubert classes. The matroid varieties that appear in the middle of the sequence of degenerations are called ``interval rank varieties,'' and even though this was not the original aim, their procedure ends up providing a way to find the cohomology class of an arbitrary interval rank variety. Our goal in this paper, which we will only partially accomplish, will be to generalize their degeneration procedure to find the classes of a larger collection of matroid varieties.

We will start by describing the degeneration procedure used in their rule:

\begin{defn}
\label{def:geomshift}
Given a closed subset $V\subseteq G(k,n)$, the \defof{geometric shift from $i$ to $j$ of $V$} is the variety \[\Sha_{i\to j}V=\lim_{t\to\infty}\exp(te_{ij})\cdot V,\] where $e_{ij}$ is the matrix whose only nonzero entry is a 1 in row $i$ and column $j$. That is, take the set of points $\{(\exp(te_{ij})v,v):v\in V\}$ in $G(k,n)\times\A^1$, take the closure inside $G(k,n)\times\P^1$, and let $\Sha_{i\to j}V$ be the fiber over $\infty$.
\end{defn}

By the definition of rational equivalence, $[V]$ and $[\Sha_{i\to j}V]$ are equal in the Chow ring of $G(k,n)$. In general, even if $V$ is irreducible, its shifts might have multiple components. The Littlewood-Richardson rule we're examining works by taking a Richardson variety and performing a prescribed sequence of geometric shifts. Every time we have multiple irreducible components, we will perform the remaining shifts on each component separately, and at the end of this process everything will have become a Schubert variety.

So it's enough to understand what happens to our varieties when we do a geometric shift. The general question, even just for matroid varieties, will prove to be very difficult, and we'll see in Counterexample \ref{cex:squarepos} that, in general, a geometric shift of a matroid variety doesn't have to be a matroid variety. However, we will be able to understand enough about the behavior of geometric shifts to make our Littlewood-Richardson rule work. We will start with a simple case:

\begin{prop}
\label{prop:shiftsingle}
Take $S\subseteq[n]$ and $0\le r<k$, and let $M$ be the matroid on $[n]$ generated (in the sense of Definition \ref{def:matroidgen}) by imposing the condition that $S$ has rank $r$. If $i\in S$ or $j\notin S$, then $\Sha_{i\to j}X(M)=X(M)$. Otherwise, $\Sha_{i\to j}X(M)=X(M')$, where $M'$ is the matroid in which $S-\{j\}\cup\{i\}$ has rank $r$ and there are no other conditions.
\end{prop}
\begin{proof}
This is \cite[6.1]{knutir}.
\end{proof}

When one performs a geometric shift on a Richardson variety, the result is almost never another Richardson variety. By picking our shifts carefully, though, we are able to make sure we stay inside a larger but still well-behaved class of matroid varieties:

\begin{defn}
\label{def:interval}
An \defof{interval} is a subset of $[n]$ of the form $\{i:a\le i\le b\}$.
\end{defn}

\begin{defn}
\label{def:irank}
An \defof{interval rank matroid} is a matroid which is generated by putting rank conditions on intervals. An \defof{interval rank variety} is the matroid variety of an interval rank matroid.
\end{defn}

Note that Schubert and Richardson varieties are themselves interval rank varieties. It's not true that all shifts of interval rank varieties are still interval rank varieties. The strategy is to instead find a specific sequence of shifts that always works. That is, starting with an interval rank variety, we perform a specific geometric shift which gives us a reduced union of different interval rank varieties of the same dimension. Then we can take each of those components and repeat the process until we only have Schubert varieties left.

\begin{lem}
\label{lem:irank}
Suppose $M$ is an interval rank matroid of rank $k$, defined by rank conditions of the form ``$I_m$ has rank $r_m$'' for intervals $I_m$. Suppose further that, for some $i,j\in[n]$, there is a unique $p$ such that $i\notin I_p$ and $j\in I_p$, and that $I_p-\{j\}\cup\{i\}$ is an interval. If $M_p$ is the matroid generated only by the condition on $I_p$ and $M'$ is the matroid generated by all the rank conditions except the one on $I_p$, then \[\Sha_{i\to j}X(M)=X(M')\cap\Sha_{i\to j}X(M_p)\] as schemes.
\end{lem}
\begin{proof}
This is proved in Section 5.3 of \cite{knutintpos}, which is forthcoming.
\end{proof}

From here, two tasks remain. We need to guarantee the existence of a pair $i,j$ as in Lemma \ref{lem:irank} for any interval rank variety, and we need to figure out what varieties we get as the components of the intersection that occurs on the last line. Both of these tasks will be accomplished by the same combinatorial object:

\begin{defn}
\label{def:irankmat}
An \defof{interval rank matrix of rank $k$} is an upper triangular $n\times n$ matrix of nonnegative integers $r_{ij}$ with the following properties:
\begin{enumerate}
\item $r_{1n}=k$.
\item $r_{i-1,j}$ and $r_{i,j+1}$ are both either $r_{ij}$ or $r_{ij}+1$. (Here we take $r_{ij}$ to be 0 whenever $i>j$.)
\item If $r_{ij}$=$r_{i-1,j}$=$r_{i,j+1}$, then $r_{i-1,j+1}=r_{ij}$.
\end{enumerate}
\end{defn}

There is a one-to-one correspondence between interval rank matroids of rank $k$ on $[n]$ and $n\times n$ interval rank matrices of rank $k$. In fact, given an interval rank matrix, imposing the condition on the Grassmannian that $\rk([i,j])\le r_{ij}$ (that is, just setting the corresponding Pl\"uckers equal to 0) is enough to define the corresponding interval rank variety, even as a scheme (\cite[1.8]{knutir}).

Many of the conditions in the interval rank matrix are redundant. For example, if we know that $\rk(\{1,2,3\})\le 1$, we don't need to also say that $\rk(\{1,2\})\le 1$ or $\rk(\{1,2,3,4\})\le 2$. So it's possible to eliminate the condition on $[i,j]$ whenever $r_{ij}=r_{i,j+1}$ or $r_{ij}=r_{i,j-1}+1$, and similarly for changing $i$. Chasing through the definitions, one can see that this is equivalent to the following:

\begin{defn}
\label{def:essentialset}
Let $(r_{ij})_{1\le i\le j\le n}$ be an interval rank matrix of rank $k$. We can use $r$ to define a \defof{partial permutation matrix} (that is, a matrix whose entries are all 0 except for at most one 1 in each row and column) by putting a 1 in position $(i,j)$ if and only if $r_{ij}=r_{i,j-1}=r_{i+1,j}\ne r_{i+1,j-1}$.

Take the partial permutation associated to an interval rank matrix, and ``cross out'' every position strictly below or to the left of a 1, and every empty row and empty column. We are left with some intact entries in the matrix, called the \defof{diagram} of the partial permutation. Some interval $[i,j]$ is called \defof{essential} if $(i,j)$ is in the upper-right corner of a connected component of the diagram.
\end{defn}

\begin{prop}
\label{prop:essential}
If $M$ is an interval rank matroid, the rank conditions coming from the essential set define $X(M)$ as scheme.
\end{prop}
\begin{proof}
This is \cite[2.3]{knutir}.
\end{proof}

This gives us a good way to choose our interval: we can find the essential set of our interval rank matroid and take, among all essential intervals starting after 1 which are tied for the rightmost right endpoint, the one with the rightmost left endpoint. If this interval is $[i,j]$, then the shift $\Sha_{i-1\to j}$ will clearly always satisfy the hypotheses of Lemma \ref{lem:irank}.

The last task is to describe what the result of this shift is. After we perform the shift, we will be left with a bunch of rank conditions on intervals inside $[n]$. We can take the pseudo-rank function generated by these conditions and try to use them to fill out an interval rank matrix. In general, we will fail: what we get won't always be an interval rank matrix because it won't satisfy property 3 in Definition \ref{def:irankmat}. So we need to figure out how to split up the subscheme of $G(k,n)$ we get by imposing these new conditions into irreducible components. This turns out to be surprisingly straightforward.

\begin{lem}
\label{lem:intersectir}
The intersection of interval rank varieties is a reduced union of interval rank varieties.
\end{lem}
\begin{proof}
This is \cite[2.2]{knutir}.
\end{proof}

\begin{prop}
\label{prop:irankscheme}
Let $M$ be an interval rank matrix, possibly not satisfying property 3 in Definition \ref{def:irankmat}. Suppose \[\begin{array}{cc}r&r+1\\r&r\end{array}\] appears somewhere in the middle of $M$. If $M_1$ is the rank matrix with that block replaced by \[\begin{array}{cc}r&r\\r&r\end{array}\] and $M_2$ is the rank matrix with that block replaced by \[\begin{array}{cc}r&r+1\\r-1&r\end{array},\] then the scheme defined by $M$ is the union of the schemes defined by $M_1$ and $M_2$.
\end{prop}
\begin{proof}
As sets, this is clear. The statement about subschemes then follows from Lemma \ref{lem:intersectir}.
\end{proof}

So now we have our procedure. Starting with a Richardson variety --- or indeed any interval rank variety at all --- we can perform a sequence of shifts, splitting the results up according to the prescription in Proposition \ref{prop:irankscheme} whenever we have more than one irreducible component. If we arrive at a point where no more shifting can be done, we must be in a position where every essential interval starts at 1, that is, we must have a Schubert variety. And every shift we do either decreases the sum of the left and right endpoints of all essential intervals (if the shift is irreducible) or expresses the cohomology class we're interested in as a nontrivial sum of smaller cohomology classes (if the shift isn't irreducible), and neither of these can continue forever. So this process must terminate eventually, and when it does we are left with a bunch of Schubert varieties with respect to the standard basis of $\C^n$.

\chapter{Positroid Varieties}
In addition to the fact that they fill in the gap between Richardson and Schubert varieties in the right way, interval rank varieties have several nice geometric properties. They share these properties with a larger class of matroid varieties, which we will describe in this section.

\begin{defn}
\label{def:cyclicinterval}
A \defof{cyclic interval} is a subset of $[n]$ which can be written as a cyclic permutation applied to an interval.
\end{defn}

\begin{prop}
\label{prop:posdefs}
If $M$ is a matroid of rank $k$ on $[n]$, the following are equivalent:
\begin{enumerate}
\item $M$ is generated by rank conditions on cyclic intervals.
\item $M$ is the matroid of an $\R$-point of the Grassmannian for which every Pl\"ucker coordinate is nonnegative.
\item $X(M)$ is the image of a Richardson variety in the flag variety $Fl(n)$ under the natural projection map $Fl(n)\to G(k,n)$.
\end{enumerate}
\end{prop}
\begin{proof}
For the equivalence of (1) and (2), see \cite{posschub}. For (2) and (3), see \cite{klspos}.
\end{proof}

\begin{defn}
\label{def:positroid}
A matroid satisfying any of the equivalent conditions just listed is called a \defof{positroid}, and its matroid variety is called a \defof{positroid variety}.
\end{defn}

Several nice properties of positroid varieties are described in \cite{projrich}. In particular, they are always reduced, irreducible, and Cohen-Macaulay, and unlike general matroid varieties (see Counterexample \ref{cex:pluckergen}) they are always cut out by Pl\"ucker variables. Positroids are very well-studied already, and there are several different combinatorial gadgets that have been invented to describe them, some of which are described in \cite{klspos}.

Similarly to how we handled interval rank varieties, we can describe a positroid by saying what the rank of every cyclic interval is. We wind up with something analogous to Definition \ref{def:irankmat}.

\begin{defn}
\label{def:posrankmat}
Take a positroid $P$ on $[n]$. We'll think of elements $[n]$ as representatives of equivalence classes of integers mod $n$ with the obvious cyclic order (that is, 1 comes right after $n$). We'll use interval notation with this in mind; for example, if $n=6$, then we'll write $[5,8]=[5,2]=\{5,6,1,2\}$. In particular, $[5,5]=\{5\}$, whereas $[5,11]=[5,10]=[5,4]=\{5,6,1,2,3,4\}$. We can form a \defof{cyclic rank matrix} by setting $r_{ij}=\rk_P([i,j])$ for any integers $i,j$ with $0\le j-i\le n$.
\end{defn}

The conditions in Definiton \ref{def:irankmat} are again necessary and sufficient for a rank matrix to have arisen from this procedure. We can also replicate the essential set machinery in this setting:

\begin{defn}
\label{def:posessential}
We can form an \defof{affine permutation matrix} from our cyclic rank matrix using the same condition we used for interval rank varieties: put a 1 in position $(i,j)$ if $r_{ij}=r_{i,j-1}=r_{i+1,j}\ne r_{i+1,j-1}$ and a 0 otherwise. Unlike in the interval rank case, every row and every column will have exactly one 1.

The \defof{essential set} is also defined exactly as before: cross out all the positions strictly below or to the left of a 1 in the partial permutation matrix, and take the positions which are at the upper-right corners of their connected components.
\end{defn}

By convention, we don't consider positions on the very upper-right edge of the matrix (that is, ones where $j-i=n$) to be essential. Again, imposing the rank conditions corresponding to the essential intervals are enough to define a positroid variety in $G(k,n)$ as a scheme.

\begin{exmp}
\label{ex:positroid}
The positroid of rank 3 on $[6]$ generated by forcing $[1,3]$, $[3,5]$, and $[5,1]$ to have rank 2 has the following cyclic rank matrix:

\[
\begin{array}{cccccccccccc}
1&2&\underline2&3&3&3&3& & & & &\\
 &1&2&3&3&\underline3&3&3& & & &\\
 & &1&2&\underline2&3&3&3&3& & &\\
 & & &1&2&3&3&\underline3&3&3& &\\
 & & & &1&2&\underline2&3&3&3&3&\\
 & & & & &1&2&3&3&\underline3&3&3\\
\end{array}
\]

We think of the matrix as repeating indefinitely in the northwest and southeast directions. So, for example, the 3 printed in the fourth row and fourth nonempty column indicates that $[4,7]=[4,1]=\{4,5,6,1\}$ has rank 3. The affine permutation corresponding to this rank matrix is the one with 1's in the spots marked in underline. We'll sometimes write affine permutations as functions, listing the images of each element of $[n]$ in order. So, for example, this one is $3,6,5,8,7,10$.
\end{exmp}

As we have already mentioned, finding the cohomology class of an arbitrary matroid variety is probably an impossible task. But given how nice positroid varieties are, it seems much more reasonable that there might be a nice way to describe their classes. There is a sense in which this has already been done in \cite{klspos}: the authors of that paper give a procedure which takes a positroid $P$ and outputs a symmetric function which represents the class of $X(P)$ in the cohomology ring of $G(k,n)$. But the symmetric function they give (called the ``affine Stanley symmetric function'') is not always a nonnegative linear combination of Schur functions. Instead, when expanded in the Schur basis all the minus signs happen to appear only in front of Schur functions which map to zero in $H^*(G(k,n))$.

It would be good to instead have a ``positive'' rule, like we had for interval rank varieties, that is, a rule that takes in a positroid and simply outputs the coefficients of the Schuberts in its cohomology class without having to go through the computationally opaque step of finding a representative for a symmetric function modulo an ideal.

Perhaps we could proceed in a way similar to our strategy for interval rank varieties: find a shift that always works, do it to get a union of positroid varieties, and continue until we have only Schuberts. The rest of this paper is dedicated to exploring the extent to which this plan can be successful.

\chapter{Complications with Shifting Positroids}
\section{The Square Positroid}

Note that the choice of which sequence of shifts to perform is much less clear in this case. Since the elements of an interval rank matroid had a natural linear order, it made sense to start at one end and perform the shortest and rightmost possible shifts in order until the end. In a positroid, of course, there is no ``rightmost,'' so any such choice would be more arbitrary than in the interval rank case.

At any rate, it doesn't matter: there is a positroid variety in $G(3,8)$, which will be defined in a moment, for which no nontrivial shift is a reduced union of positroid varieties. So if we are going to use geometric shifts to find the class of a positroid variety, it won't be as easy as it was for interval rank varieties.

\begin{figure}[t]
\begin{center}\includegraphics[width=6cm]{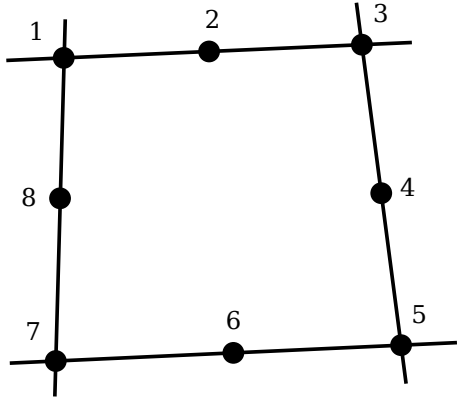}\end{center}
\caption{A projective model of the ``square positroid.''}
\label{fig:squarepos}
\end{figure}

\begin{cex}
\label{cex:squarepos}
The \defof{square positroid} is a rank-3 positroid $S$ on $[8]$. Its essential intervals are $[1,3]$, $[3,5]$, $[5,7]$, and $[7,1]$, each of which has rank 2. A picture of a projective model of $S$ is shown in Figure \ref{fig:squarepos}.

The picture makes it clear that this matroid has $D_4$ symmetry. Table \ref{tab:shiftsquare} is a catalog of the results of performing every possible geometric shift on $X(S)$ up to symmetry. Each entry in the table is a list of the irreducible components of the corresponding shift. In this table and in the other similar tables later in this paper, we will describe a matroid by listing the nonredundant conditions on its rank function; each row of the table corresponds to one connected component. We will write, for example $(123)_2$ to mean that the set $\{1,2,3\}$ has rank 2. The rank of any set not listed is understood to be the largest possible number that doesn't violate the condition $\rk(F\cup\{x\})=\rk F\mbox{ or }\rk F+1$. In this way, $S$ itself could be described by the conditions \[(123)_2\ (345)_2\ (567)_2\ (781)_2.\]
\begin{table}[t]
\label{tab:shiftsquare}
\begin{center}
\begin{tabular}{c|r}
\textbf{shift}&\textbf{result}\\\hline
$1\to 2$&no change\\\hline
$1\to 3$&$(123)_2\ (145)_2\ (567)_2\ (781)_2$\\\hline
$1\to 4$&$(1235)_2\ (567)_2\ (781)_2$\\
&$(13)_1\ (567)_2\ (781)_2\ (783)_2$\\
&$(1)_0\ (567)_2$\\\hline
$1\to 5$&$(1234)_2\ (6781)_2$\\
&$(13)_1\ (6781)_2\ (6783)_2$\\
&$(17)_1\ (1234)_2\ (7234)_2$\\
&$(137)_1$\\
&not a matroid variety, described in the text\\\hline
$2\to 3$&$(123)_2\ (245)_2\ (567)_2\ (781)_2$\\\hline
$2\to 4$&$(1235)_2\ (567)_2\ (781)_2$\\
&$(23)_1\ (567)_2\ (781)_2$\\\hline
$2\to 5$&$(1234)_2\ (267)_2\ (781)_2$\\
&$(23)_1\ (267)_2\ (367)_2\ (781)_2$\\
&$(2)_0\ (781)_2$\\\hline
$2\to 6$&$(123)_2\ (257)_2\ (345)_2\ (781)_2$
\end{tabular}
\end{center}
\caption{The results of the possible geometric shifts of $X(S)$ up to symmetry.}
\end{table}
Some features of the table are worth pointing out explicitly. First, except for the shift $1\to 2$, which does nothing, for none of the shifts is it the case that every irreducible component is a positroid variety. This means that the naive goal of just replacing the words ``interval rank variety'' with ``positroid variety'' everywhere in Section \ref{sec:irank} and hoping to find the right shift order cannot possibly work --- there is no way to stay entirely inside the world of positroid varieties using geometric shifts.

Also worth discussing is the fifth component of the shift $1\to 5$ in the table, which we'll call $B$. As indicated there, $B$ is not a matroid variety. Its ideal in the coordinate ring of the Grassmannian is generated by all the Pl\"ucker coordinates that contain 1 (which exactly forces 1 to be a loop) along with the single cubic binomial \[p_{234}p_{367}p_{578}-p_{235}p_{347}p_{678}.\] This ideal is prime, and the matroid of a generic point of $B$ is simply the matroid in which 1 is a loop.
\end{cex}

The preceding counterexample does much to dash our hopes of a geometric-shift-based way to find the cohomology class of a positroid variety. However, at least in rank 3, all is not lost. Look at the entry in Table \ref{tab:shiftsquare} for the shift $2\to 4$. The second component listed is definitely a positroid variety, since all the conditions are on cyclic intervals. The first component is not: one of the conditions is on the set $\{1,2,3,5\}$. But notice that the number 4 doesn't appear in that section of the table; there are no conditions placed on that element of the matroid at all. And if 4 is deleted from the matroid, $\{1,2,3,5\}$ becomes a cyclic interval and we have a positroid again.

So while the first component of the shift $2\to 4$ isn't a positroid variety, it is the variety of a free extension of a positroid (see Definition \ref{def:freeext}). From a certain perspective, it's not so strange that it worked out this way. In the original description of $S$, 4 appears in only one of the essential rank conditions --- the one on $\{3,4,5\}$. This is exactly the sort of situation that made us happy when we were working with interval rank varieties.

\section{Partial Success in Rank 3}
In fact, for positroids in rank 3 we can always arrange for this to happen. Since we will be making use of free extensions and Proposition \ref{prop:freeext}, it will be simpler to work in the equivariant cohomology ring $H^*_{GL_3}(\Mat_{3\times n})$ for this computation. The geometric shift is defined in exactly the same way this context, and we can extract results about $H^*(G(3,n))$ by restricting classes inside $\Mat_{3\times n}$ to the open subvariety $\Mat_{3\times n}^\circ$ consisting of full-rank matrices, and using the fact that \[H^*_{GL_3}(\Mat_{3\times n}^\circ)\cong H^*(\Mat_{3\times n}^\circ/GL_3)=H^*(G(3,n)).\]

\begin{thm}
\label{thm:rank3pos}
If $P$ is a positroid of rank 3, there is a geometric shift we can apply to $V(P)$ which, up to cohomology, results in a reduced union of positroid varieties and free extensions of positroid varieties.
\end{thm}

Before we prove this, we first analyze the essential set of a positroid $P$ of rank 3. The essential conditions will, of course, be on sets of rank 0, 1, or 2. By repeatedly using Corollary \ref{cor:loopcoloop}, we may reduce to the case where $P$ has no loops, that is, there are no nonempty subsets of rank 0.

So all essential conditions must be on sets of rank 1 or 2. Conditions of rank 1 (that is, conditions which force elements of the matroid to be parallel) put an equivalence relation on $[n]$, and any essential rank-2 interval must be a union of these equivalence classes. This is because if some interval $I$ has rank 2 and some interval $J$ has rank 1, then there are two nonisomorphic ways to resolve this: either $I\cup J$ has rank 2 or $I\cap J$ has rank 0. (One can deduce this formally by repeatedly applying Definition \ref{def:posessential}.) This is a problem unless $I\cap J$ is empty or $J\subseteq I$, that is, $I$ contains the entire equivalence class $J$ or none of it.

Furthermore, the intersection of two essential intervals must have rank at most 1: if $I$ and $J$ have rank 2, and $\rk(I\cap J)>0$, then either $\rk(I\cap J)=1$ or $\rk(I\cup J)=2$. But in the latter case, $I$ and $J$ wouldn't have been essential intervals, because the rank conditions on $I$ and $J$ would be implied by the condition on $I\cup J$. This means that, provided $I$ and $J$ intersect in an interval (which must happen unless $I\cup J=[n]$), $I\cap J$ consists of a single equivalence class.

\begin{figure}[t]
\begin{center}\includegraphics[width=6cm]{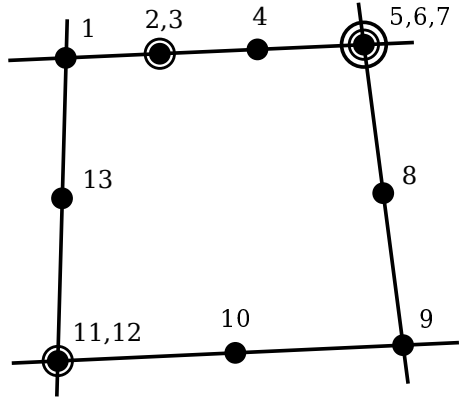}\end{center}
\caption{An example of a ``model polygon'' of a positroid}
\label{fig:modelpoly}
\end{figure}

So we can think of a projective model of $P$ as a polygon, which we will call the \defof{model polygon}, with points drawn at all the vertices and at various points along the edges. Each edge of the polygon is a rank-2 essential interval and each point is a rank-1 essential interval (which may consist of multiple parallel elements of $[n]$). We'll make use of this description throughout the section.

For example, Figure \ref{fig:modelpoly} shows a picture of a model polygon of a positroid. In this example, $n=13$; there are four edges (rank-2 essential intervals): $[1,7]$, $[5,9]$, $[9,12]$, and $[11,1]$; and there are three rank-1 essential intervals: $[2,3]$, $[5,7]$, and $[11,12]$.

We can immediately eliminate the case where there are only two rank-2 essential intervals which intersect on both ends. This happens, for example, in the rank-3 positroid on $[6]$ in which $[1,4]$ and $[4,1]$ are the only essential intervals, each of rank 2. At every point of the corresponding open positroid variety, 1 and 4 are parallel, and something similar happens in general. Suppose $[n]=A\sqcup B\sqcup C\sqcup D$ with $A,B,C,D$ appearing cyclically consecutively. The case we're interested in is where our rank-2 essential sets are, say, $A\cup B\cup C$ and $C\cup D\cup A$. But by the logic from two paragraphs ago, this means that, on the open set $\Mat_{3\times n}^\circ$ of full-rank matrices, $A\cup C$ has rank 1. So we could reorder our points in the order $B,A,C,D$, and our positroid turns out to actually be an interval rank matroid.

To prove Theorem \ref{thm:rank3pos}, we'll need some lemmas:

\begin{lem}
\label{lem:transverse}
Let $u,v\in[n]$. Suppose $M$ is a free extension of a positroid of rank 3 in which all rank conditions are on cyclic intervals in $[u,v]$, and in which $u$ and $v$ are each only in one essential interval, which has rank 2. Let $M'$ be a matroid with the same conditions, except having rank conditions on cyclic intervals in $[v,u]$. Then $V(M)$ and $V(M')$ intersect generically transversely in $\Mat_{3\times n}$.
\end{lem}

The proof of this is somewhat involved, and it will be helpful to establish some more lemmas beforehand.

\begin{lem}
\label{lem:transversenonsingular}
Suppose $M$ and $M'$ satisfy the hypotheses of Lemma \ref{lem:transverse}. Then the generic point of $V(M)\cap V(M')$ is a nonsingular point of both $V(M)$ and $V(M')$.
\end{lem}
\begin{proof}
Since $u$ and $v$ each appear in at most two essential intervals, each of rank 2, at a generic point of this intersection, the $u$'th and $v$'th columns of the matrix are nonzero. And the dimension of the tangent space of $V(M)$ is the same at any point at which the $u$'th and $v$'th columns are nonzero, since there is an automorphism interchanging any two such points, so each such point is a nonsingular point of $V(M)$.
\end{proof}

\begin{lem}
\label{lem:transversetangentexistence}
Under the hypotheses of Lemma \ref{lem:transverse}, let $A\in V(M)\cap V(M')$, and pick a tangent vector $Q$ at $A$ in $\Mat_{3\times n}$. Then there is a tangent vector to $V(M)$ which agrees with $Q$ in the columns in $[v,u]$.
\end{lem}
\begin{proof}
First, if one or both of $M$ or $M'$ is a free extension of a positroid, then there is a column which is free in both $V(M)$ and $V(M')$. So it's enough if we can get any tangent vector at $A$ in the projection of the tangent space away from that column, since we can then fix it in any way we want without leaving either $V(M)$ or $V(M')$. So it's enough if we handle the case where $M$ and $M'$ are actually positroids.

We think of $Q$ as a $\C[\epsilon]/(\epsilon^2)$-point of $\Mat_{3\times n}$, that is, a $3\times n$ matrix with entries in $\C[\epsilon]/(\epsilon^2)$. We're looking for another such matrix $W$ which lies in $V(M)(\C[\epsilon]/(\epsilon^2))$. Like $Q$, $W$ should map to $A$ after setting $\epsilon$ to 0, and it should agree with $Q$ on the columns in $[v,u]$.

We will determine the entries of $W$ column by column. For $i\in[n]$, write $q_i$, $w_i$, and $a_i$ for the $i$'th columns of $Q$, $W$, and $A$ respectively. As we mentioned above, at a generic point of the intersection, we can assume that $a_u$ and $a_v$ are nonzero. First, $V(M)$ imposes no conditions on the columns in $[v,u]$ except $u$ and $v$ themselves, so we are free to set those columns of $W$ to whatever we like without leaving $V(M)$. So we simply need to verify that there is a matrix with the right entries in columns $u$ and $v$ that lies in $V(M)$ and still maps to $A$.

Following the ``model polygon'' description of rank-3 positroids from above, let $I$ be the edge (that is, rank-2 essential condition) containing $u$ and let $J$ be the edge containing $v$. Let $u'$ be the rightmost (that is, further from $u$) element of $I$ so that $a_{u'}$ isn't a multiple of $a_u$, or set $u=u'$ if there is no such element. Similarly, either let $v'$ be the leftmost element of $J$ with $a_{v'}$ not a multiple of $a_v$ or set $v=v'$.

If $I\ne J$, then we can construct $W$ by making the columns in $[v,u]$ agree with those of $Q$, as we must, and letting columns $u'$ and $v'$ be as in $A$, that is, have 0 as their coefficient of $\epsilon$. We similarly let any columns outside of $I\cup J$ be as in $A$. By construction, $a_u$ and $a_{u'}$ span the columns in $I$. Given an $x$ with $u<x<u'$, say $a_x=\alpha a_u+\beta a_{u'}$. We set $w_x=\alpha w_u+\beta w_{u'}$. After doing this, we see we have accomplished our goal: all the essential rank conditions defining $V(M)$ have been met by $W$, and by construction, $W$ maps to $A$ after killing $\epsilon$.

Finally, if $I=J$, we can do the same thing, except that since $u'=v$ and $v'=u$, we can't force those columns to agree with $A$. But in this case, the entire interval $[u,v]$ is a single essential rank-2 interval, so we can simply set $w_u$ and $w_v$ equal to $q_u$ and $q_v$ and handle the interior columns as in the preceding paragraph.
\end{proof}

\begin{proof}[Proof of Lemma \ref{lem:transverse}]
Take a generic point of $V(M)\cap V(M')$. By Lemma \ref{lem:transversenonsingular}, this is a nonsingular point of both $V(M)$ and $V(M')$. The tangent space at a point of $\Mat_{3\times n}$ is isomorphic to $\Mat_{3\times n}$, so if $A$ is a generic element of $V(M)\cap V(M')$, it's enough to show that the tangent spaces of $V(M)$ and $V(M')$ at $A$ together span all of $\Mat_{3\times n}$. We'll show that the projection of the tangent space at $A$ of $V(M)$ to the subspace of $\Mat_{3\times n}$ spanned by the columns in $[v,u]$ is surjective. This, together with the symmetrical claim about $V(M')$ and $[u,v]$, is enough to prove the result. Take a tangent vector at our intersection point. By Lemma \ref{lem:transversetangentexistence}, there is a tangent vector to $V(M)$ which agrees with it in the columns in $[v,u]$. But since $V(M')$ imposes no conditions on the columns in $(u,v)$, we can set the columns in $(u,v)$ to whatever we want by adding a vector tangent to $V(M')$ without disturbing anything else.
\end{proof}

\begin{lem}
\label{lem:widecomponents}
For any $GL_3$-invariant equidimensional subvariety $V\subseteq\Mat_{3\times n}$, if $V$ has components contained in $L$, then we claim its equivariant cohomology class when expanded in the Schur basis will contain some terms corresponding to partitions with more than $n-3$ columns.
\end{lem}
\begin{proof}
First, all the coefficients in the Schur basis are nonnegative: if some $\sigma_\lambda$ appears with a negative coefficient, then free-extend enough times so that $n-3$ is at least the number of columns in $\lambda$. Then removing $L$ and passing to the Grassmannian gives us a negative coefficient in the Schur basis there, which can't happen. So if there is a component of $V$ contained in $L$, then it maps to 0 after passing to the Grassmannian, so its class lies in the kernel of that map, that is, it's a nonnegative sum of terms corresponding to partitions with more than $n-3$ columns.
\end{proof}

\begin{lem}
\label{lem:nobadcomponents}
Suppose the model polygon of $P$ doesn't have two edges which overlap on both ends, and that at least two of the corners don't have any parallel points. Then $V(P)$ is generically equal to the variety cut out in $\Mat_{3\times n}$ by the rank conditions corresponding to its essential intervals.
\end{lem}
\begin{proof}
Let $Y(P)$ be the subvariety of $\Mat_{3\times n}$ cut out by the rank conditions corresponding to the essential intervals of $P$. Since we know $Y(P)=X(P)$ in $G(3,n)$, and therefore in $\Mat_{3\times n}^\circ$, it's enough to check that $Y(P)$ has no components contained in $L:=\Mat_{3\times n}-\Mat_{3\times n}^\circ$.

Let $E_1,\ldots,E_k$ be the subvarieties corresponding to each of the rank conditions on the edges of the model polygon, except with only one element of each rank-1 essential interval included. In the example in Figure \ref{fig:modelpoly}, $k$ would be 4, and the corresponding sets could be $\{1,2,4,5\}$, $\{7,8,9\}$, $\{9,10,11\}$, and $\{12,13,1\}$. Let $A_1,\ldots,A_l$ be the ones corresponding to the rank-1 essential conditions. In the example, $l=3$ and the sets are $\{2,3\}$, $\{5,6,7\}$, and $\{11,12\}$.

We know from Lemma \ref{lem:transverse} that the intersection $E_1\cap\cdots\cap E_k$ is generically transverse, and we claim the same is true of each $(E_1\cap\cdots\cap E_k\cap A_1\cap\cdots A_{i-1})\cap A_i$. To see this, suppose $A_i$ is forcing $x_1,x_2,\ldots,x_r$ to be parallel, and that $x_1$ is the one with a condition imposed on it by the $E$'s. Then at any point of the intersection, columns $x_2,\ldots,x_r$ are free in the tangent space to $E_1\cap\cdots\cap E_k\cap A_1\cap\cdots A_{i-1}$, all the other columns are free to vary in the tangent space to $A_i$, so together their tangent spaces span.

Write \[B=E_1\cap\cdots\cap E_k\cap A_1\cap\cdots\cap A_l.\] We then have $V(P)\subseteq Y(P)\subseteq B$, and we claim that they are all equidimensional of the same dimension. The fact that $\dim V(P)=\dim B$ follows either from direct computation or from the expected codimension machinery developed in the next section. We know that $V(P)$ is equidimensional (in fact, irreducible) from the positroid machinery, and $B$ is equidimensional because it's the generically transverse intersection of Schuberts. Finally, the hypothesis on the corners of the model polygon of $P$ allow us to conclude via Lemma \ref{lem:transverse} that $Y(P)$ is the generically transverse intersection of two interval rank varieties.

So it is sufficient to show $B$ has no components contained in $L$. By Lemma \ref{lem:widecomponents}, this is the same as showing that no terms of $B$ correspond to partitions with more than $n-3$ columns. But we can compute the class of $B$ directly. From its description as a transverse intersection, its class is \[\sigma_{e_1-2}\cdots\sigma_{e_k-2}\sigma_{(a_1-1)^2}\cdots\sigma_{(a_l-1)^2},\] where $e_i$ is the number of points on edge $i$ of the model polygon, and $a_i$ is the number of points in the $i$'th rank-1 essential interval. By the Pieri rule, the number of columns in each term of this is at most \[\sum e_i+\sum a_j-2\#(\mbox{edges})-\#(\mbox{rank-1 intervals}).\] Some simple counting shows that \[n=\sum e_i+\sum a_j-\#(\mbox{rank-1 intervals})-\#(\mbox{corners}).\] If $P$ is an interval rank matroid, this follows directly from that theory, so we may assume that the model polygon of $P$ is closed, with as many corners as edges. Since we've eliminated the case of two edges by hypothesis, the number of edges is at least 3. All together, this says that when we expand the class of $B$ in the Schur basis, each term corresponds to a partition with at most $n-3$ columns, which gives the result.
\end{proof}

\begin{lem}
\label{lem:parshift}
For any matroid $M$ on $[n]$ in which the point $k$ is not a loop, write $M^k$ for the matroid on $[n+1]$ obtained by taking the free extension of $M$ in which the new element is $n+1$ and adding the condition that $n+1$ is parallel to $k$. (We call this a \defof{parallel extension} of $M$.) Fix $i,j\neq k$. Suppose that $\Sha_{i\to j}V(M)=\bigcup_rV(N_r)$ for some matroids $N_r$. Then $k$ is not a loop in any $N_r$, and \[\Sha_{i\to j}V(M^k)=\bigcup_rV((N_r)^k).\]
\end{lem}
\begin{proof}
For a subvariety $V\subseteq\Mat_{3\times n}$, write $\Sigma_{i\to j}^\circ V$ for the variety \[\{(\exp(te_{ij})v,v):v\in V,t\in\A^1\}\] in $\Mat_{k\times n}\times\A^1$, and write $\Sigma_{i\to j}V$ for its closure in $\Mat_{3\times n}\P^1$. Then by definition $\Sha_{i\to j}V$ is the fiber over $\infty$ of $\Sigma_{i\to j}V$.

Let $\pi:\Mat_{3\times(n+1)}\to\Mat_{3\times n}$ be the map that drops the $(n+1)$'st column, and let $A$ be the rank-$3$ matroid on $[n+1]$ in which $k$ and $n+1$ are parallel with no other conditions. For any subvariety $V\subseteq\Mat_{3\times n}$, we can define $V^k\subseteq\Mat_{3\times (n+1)}$ to be $\pi^{-1}(V)\cap V(A)$. We can define the same operation on subvarieties of $\Mat_{3\times n}\times\A^1$ or $\Mat_{3\times n}\times\P^1$ by just working on the first component. Then away from the locus where $k$ is a loop, $V(M^k)=V(M)^k$. Furthermore, \[\Sigma_{i\to j}^\circ (V^k)=(\Sigma_{i\to j}^\circ V)^k\] for any $V$: applying $\Sigma_{i\to j}^\circ$ replaces an equation $f(M)=0$ by the equation $f(\exp(-te_{ij})M)=0$, but because $i,j\ne k$, the equations defining $V(A)$ are invariant under this operation. Therefore, we can conclude that their closures are equal as well: \[\Sigma_{i\to j}(V^k)=(\Sigma_{i\to j}V)^k.\]

We claim that none of the components of $\Sha_{i\to j}V(M)$ or $\Sha_{i\to j}V(M^k)$ live inside the locus where $k$ is a loop. If there was such a component, say $C\subseteq \Sha_{i\to j}V(M)$, then consider the inclusion $\Mat_{3\times(n-1)}\inj\Mat_{3\times n}$ and write $L$ for its image. Then in fact $C$ lies in $\Sha_{i\to j}(V(M)\cap L)$: as in the last paragraph, $\Sigma_{i\to j}^\circ(V(M)\cap L)=(\Sigma_{i\to j}^\circ V(M))\cap(L\times\P^1)$ because the equations defining $L$ are invariant under the operation $f(M)\mapsto f(\exp(-te_{ij})M)$. But this is a contradiction: the dimension of each component of $\Sha_{i\to j}(V(M)\cap L)$ is equal to the dimension of $V(M)\cap L$, which is strictly less than $\dim V(M)=\dim C$.

So it's enough to show that $\Sha_{i\to j}(V(M)^k)=(\Sha_{i\to j}V(M))^k$, since we can then remove all the components inside the locus where $k$ is a loop to get our result. But this is clear: $\Sha_{i\to j}(V(M)^k)$ is cut out in $\Mat_{3\times(n+1)}\times\P^1$ by the equations defining $\Sigma_{i\to j}(V(M)^k)$ and $t=\infty$, but by the result from the preceding paragraph, these is the same as the equations defining $\Sigma_{i\to j}V(M)$ and $V(A)$, together with $t=\infty$. But this defines $(\Sha_{i\to j}V(M))^k$.
\end{proof}

\begin{proof}[Proof of Theorem \ref{thm:rank3pos}]
We first eliminate the case where there are parallel elements at points of the model polygon other than corners. Say $i,i+1,\ldots,i+r$ are parallel, and that they correspond to a point on the interior of an edge $A$ of the polygon. Either there are some rank-1 essential intervals between $i+r$ and the next corner, say also containing the edge $B$, or $i+r+1$ is a point on that corner. Either way, we will do the shift $i+r+1\to i$. Let $H$ be the matroid defined by the conditions on $A$ and $B$, and let $H'$ be the matroid defined by the rest of the conditions in $P$. By applying Lemma \ref{lem:parshift}, we can assume that the two points involved in the shift are the only ones with any parallel elements. Then by Lemma \ref{lem:nobadcomponents}, we know $V(P)=V(H)\cap V(H')$.

Since $\Sha_{i+r+1\to i}$ fixes $V(H')$, we want to be comparing \[X_1=\Sha_{i+r+1\to i}(V(H)\cap V(H'))\] and \[X_2=(\Sha_{i+r+1\to i}V(H))\cap V(H').\] Now, $H$ is an interval rank matroid --- we've already eliminated the case where $A$ and $B$ are the only edges of the polygon and they meet at both endpoints. So we can appeal to the interval rank machinery to conclude that $\Sha_{i+r+1\to i}V(H')$ is again a union interval rank varieties, meaning that the $X_2$ is a union of positroid varieties. So by Lemma \ref{lem:transverse}, we know that both intersections in that containment are generically transverse, so we can conclude that they have the same cohomology class.

So it's enough to understand the case where the only parallel elements are at the intersections of rank-2 essential intervals. And using Lemma \ref{lem:parshift}, as long as none of our shifts involve these corners, it is then enough to understand the case where there are no rank-1 essential intervals at all.

Let $I$ and $J$ be two cyclically adjacent rank-2 essential intervals overlapping in $\{i\}$, say with $J$ occurring to the right of $I$. (Note that if $I$ and $J$ are adjacent and disjoint, then our positroid is actually an interval rank matroid.) There must be at least three elements in $I$ and in $J$, because otherwise the condition that they have rank 2 wouldn't be essential. In the absence of parallel elements, essential intervals are only allowed to overlap in one element, so $J$ is the only essential interval containing $i+1$ and $I$ is the only essential interval containing $i-1$.

Let $E$ be the subvariety of $\Mat_{k\times n}$ defined by the rank conditions on $I$, $E'$ the subvariety defined by the rank conditions on $J$, and $S$ the subvariety defined by all the other conditions. By Lemma \ref{lem:nobadcomponents} again, we know that $V(P)=E\cap E'\cap S$. Since the shift from $i+1$ to $i-1$ fixes $E$ and $S$, we already know that \[\Sha_{i+1\to i-1}(E\cap E'\cap S)\subseteq(\Sha_{i+1\to i-1}E)\cap E'\cap S.\] The intersection $(\Sha_{i+1\to i-1}E)\cap E'$ is equal to $\Sha_{i+1\to i-1}(E\cap E')$, and it's the union of two components: $A$, in which $I\cup J-\{i+1\}$ has rank 2, and $B$, in which $i$ and $i-1$ are parallel, and $I$ and $J-\{i+1\}\cup\{i-1\}$ have rank 2. (This can either be computed directly or by using the interval rank machinery.)

Let $u$ be the left endpoint of $I$ and $v$ be the right endpoint of $J$. By Lemma \ref{lem:nobadcomponents}, the intersections $A\cap S$ and $B\cap S$ are free extensions of positroid varieties --- after removing $i+1$, $A$ is defined by a condition on the interval $[u,v]$ and $B$ is defined by conditions on $[u,i]$, $[i-1,v]$, and $[i,i-1]$. So it will be enough to show that the containment above is actually an equality up to cohomology. In fact, we claim that both $E\cap E'$ and $(\Sha_{i+1\to i-1}E)\cap E'$ intersect $S$ generically transversely, which is enough. This follows from Lemma \ref{lem:transverse}.
\end{proof}

This procedure is quite a bit more complicated than the procedure we saw for interval rank varieties, and it relies critically on our ability to isolate a single element which is in only one essential interval to be the target of our shift. As the next section will show, this situation is special to rank $\le 3$. In general, positroid varieties in rank 4 and above are even less amenable to our geometric shift operation.

\section{The Tripod Positroid and Failure in Rank 4}
Our program for applying geometric shifts to positroid varieties in rank 3 relied on being able to find some pair of elements $i,j\in[n]$ for which there was exactly one essential interval containing $j$ but not $i$. We found that, with some difficulty, we can always arrange for this to be the case in rank 3. In rank 4, though, we will not be so lucky:

\begin{figure}[t]
\begin{center}\includegraphics[width=6cm]{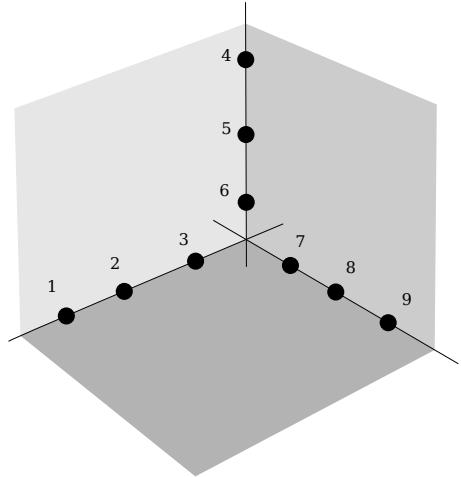}\end{center}
\caption{A projective model of the tripod positroid.}
\label{fig:tripodpos}
\end{figure}
\begin{cex}
\label{cex:tripodpos}
The \defof{tripod positroid} is a rank-4 positroid $T$ on $[9]$. Its essential intervals are $[1,6]$, $[4,9]$, and $[7,3]$ with rank 3, and $[1,3]$, $[4,6]$, and $[7,9]$ with rank 2. A projective model of $T$ is shown in Figure \ref{fig:tripodpos}.

It is immediate from the description given in the previous paragraph that there is no pair of elements $i,j$ for which there is exactly one essential interval containing $i$ but not $j$. In fact, there is enough symmetry here that there are only two essentially different shifts: ones involving two elements in the same rank-2 essential interval, and the rest.

The first kind clearly does nothing. To examine the second kind, we can pick any suitable pair of numbers and see what happens. Consider the shift $6\to 9$. We wind up with two components. Using the notation of Counterexample \ref{cex:squarepos}, they are:
\begin{align*}
A:&\quad(123)_2\ (45678)_2\ (12345678)_3\\
B:&\quad(1236)_2\ (456)_2\ (678)_2\ (123456)_3\ (456789)_3\ (678123)_3
\end{align*}
Component $A$ is a positroid variety, but component $B$ is not. In fact, we can say even more: notice that there are three different rank-2 flats in our list that contain 6. Since one can't have three different cyclic intervals which overlap in just one point, we see that the sets in $B$ can't become cyclic intervals even if we reorder the base set. That is, $B$ isn't a positroid variety with respect to any ordering of $[n]$.

Since this was the only shift we could do up to symmetry, we see that there's no way to shift $X(T)$ to get even free extensions of positroids, like we had in rank 3. It's possible, though it seems unlikely, that $B$ belongs to some class of matroid varieties larger than positroid varieties for which it's still possible to concretely describe the effects of geometric shifts.
\end{cex}

\chapter{Expected Codimension of Matroid Varieties}
Matroid varieties can be very ugly in general. A good start toward understanding the behavior of a matroid variety would be to find some way to compute its dimension directly from the matroid that defines it. Even this has very little hope of succeeding. Matroid varieties are under no obligation to be irreducible, equidimensional, or even generically reduced (when given the natural scheme structure), and, as mentioned before, even the problem of determining whether a given matroid variety is empty or not is NP-hard (\cite{stretch}).

Still, one can come up with an estimate of the codimension of a matroid variety inside its Grassmannian by keeping careful track of the conditions imposed by the vanishing of Pl\"ucker coordinates on the columns of the $k\times n$ matrix defining a point on $G(k,n)$. This section is about a way to make this idea precise, producing a number called the \defof{expected codimension} for each matroid. While it doesn't always produce the actual codimension of the matroid variety, we will prove that it always does for positroids.

The procedure we follow was described quite informally in \cite[3.3]{fnrmat}, and we flesh it out here. Consider the matroid $S$ of a point in $G(3,8)$ for which $p_{123}$, $p_{345}$, $p_{567}$ and $p_{178}$ are the only Pl\"ucker variables that vanish. A projective model for $S$ using points in $\P^2$ is shown in Figure \ref{fig:squarepos}.

We estimate the codimension of $X(S)$ in $G(3,8)$ as follows. To build a projective model of $S$ like the one in the figure, we are free to place the odd-numbered points wherever we want. Once we've done this, each even-numbered point is forced to live in the codimension-1 subspace spanned by two of the points we've already placed. So we guess that the codimension of $X(S)$ is $1+1+1+1=4$.

This turns out to be the correct answer for $\codim X(S)$, and we'll see later that the reasoning given is more or less why. One immediate question is whether the result of this procedure depends on the order in which we ``place'' the points. Once we've nailed down exactly what the procedure is, we will see that the answer to this question is no. For now, let's just try a couple more. If they're placed in order starting from the beginning, points 1, 2, 4, 5, and 7 can be put anywhere without restriction. As before, points 6 and 8 are now forced onto codimension-1 subspaces. But point 3 is now forced onto a codimension-2 subspace: it needs to be on the intersection of $\mspan\{1,2\}$ and $\mspan\{4,5\}$. So, adding all the restrictions up, we get $1+1+2=4$. Similarly, we could get ``$2+2$'' by placing the points 1, 2, 4, 5, 6, and 8 freely, and then putting 3 and 5 in last.

We will show that our definition is independent of the order by recasting it in terms of something manifestly order-independent. In $G(k,n)$, specifying exactly which Pl\"ucker coordinates vanish is the same as describing, for each subset of the set of columns, the dimension of its span in $\C^k$; in matroid language, this is called the \defof{rank} of the corresponding subset $[n]$. Our current procedure is to ask, for each element, what constraints are put on that element when it's added in. Instead let's ask, for each subset of the base set of the matroid, what constraints it puts on its elements. For example, in the set $\{1,2,3\}$ in $S$, the third element added in will be forced onto a codimension-1 subspace no matter what the order is; the only thing that matters is that the number of elements of this set is 1 more than its rank.

So it seems like we should add up the numbers $(\#F-\rk F)(k-\rk F)$ for each subset $F$; the first factor is the number of elements which are constrained by $F$ and the second is the codimension of the subspace those elements are constrained to. But this is not quite right: whenever an element belongs to two different such $F$'s, it's going to be counted twice. Sometimes this is desirable, as we saw with point 3 two paragraphs up, but often it will be redundant, as it is for the sets $\{1,2,3\}$ and $\{1,2,3,4\}$ in $S$. We ought to subtract 1 from the number of constrained elements for the larger set to account for the fact that it was already taken care of by the smaller one.

This, finally, takes us to the definition that we'll be using:

\begin{defn}
\label{def:expdim}
Let $M$ be a matroid of a point of $G(k,n)$, and let $\SS\subseteq\mathcal{P}([n])$. For $S\in\SS$, we define \[c(S)=\#S-\rk S,\] and \[a_{\SS}(S)=c(S)-\sum_{\substack{T\in\SS\\T\subsetneq S}}a_{\SS}(T),\quad a_\SS(\varnothing)=0,\] where the sum goes over elements of $\SS$. (Note that this indeed recursively defines $a$ for all elements of $\SS$.) We then define the \defof{expected codimension of $M$ with respect to $\SS$} to be \[\ec_{\SS}(M):=\sum_{S\in\SS}(k-\rk S)a_{\SS}(S).\] The \defof{expected codimension of $M$} is then \[\ec(M):=\ec_{\mathcal{P}(E)}(M).\] (Similarly, we will write $a=a_{\mathcal{P}([n])}$.) We say that $M$ \defof{has expected codimension} if $\ec(M)$ is equal to the codimension of $X(M)$ in $G(k,n)$.
\end{defn}

Allowing $\SS$ to be something other than $\mathcal{P}([n])$ itself might seem strange, but it will turn out to be very helpful. We will show that in many cases $\ec_{\SS}$ will be the same for many different choices of $\SS$ but easier to compute for some choices than for others, and we will be happy to have the flexibility, for both theoretical and practical reasons.

We will show in Theorem \ref{thm:posexpcodim} that positroids have expected codimension. In Section \ref{sec:valuativity} we also discuss \defof{valuativity}, a well-studied property of some numerical invariants of matroids, and show that expected codimension is valuative.

\section{Properties of Expected Codimension}
We now study how $\ec_\SS$ chages for different choices of $\SS$. Throughout this section, $M$ is a matroid of rank $k$ on a set $E$.

First, it will be helpful to write $\ec$ in a more symmetrical way. Thinking of $\SS$ as a poset under containment, write $\mu_{\SS}$ for its M\"obius function. Then the fact that \[c(S)=\sum_{T\subseteq S,\ T\in\SS}a_{\SS}(T)\] tells us that we can write \[a_{\SS}(S)=\sum_{T\in\SS} c(T)\mu_{\SS}(T,S),\] which means that \[\ec_{\SS}(M)=\sum_{S,T\in\SS}c(T)(k-\rk S)\mu_{\SS}(T,S).\] (Note that this is the same as summing over only the pairs $S,T$ with $T\subseteq S$, since if $T\not\subseteq S$, $\mu(T,S)=0$.) From this perspective, it seems natural to define a version of $a_{\SS}$ which splits up the sum the other way, that is, we define \[b_{\SS}(T)=\sum_{S\in\SS}(k-\rk S)\mu_{\SS}(T,S),\] and from here we may clearly write \[\ec_{\SS}(M)=\sum_{T\in\SS}c(T)b_{\SS}(T).\]

A small advantage of singling out $b$ is that it clarifies the behavior of these operations under dualization:

\begin{prop}
If $\SS\subseteq\mathcal{P}(E)$ is some collection of sets, let $\SS'=\{E-S:S\in\SS\}$. Then:
\begin{enumerate}
\item $\ec_{\SS}(M)=\ec_{\SS'}(M^*)$
\item For $S\in\SS$, $a_{\SS}(S)=b_{\SS'}(E-S)$, where the latter is computed in $M^*$.
\end{enumerate}
\end{prop}
\begin{proof}
The rank of $E-S$ in $M^*$ is $\#(E-S)-k+\rk_MS$. So $c(E-S)$ in $M^*$ is $k-\rk_MS$, and $c(S)$ in $M$ is $k-\rk_{M^*}(E-S)$. So since \[\ec_{\SS}=\sum_{S,T\in\SS}c(T)(k-\rk S)\mu_{\SS}(T,S),\] (1) follows from the fact that $\SS'$ is the opposite poset to $\SS$, and therefore $\mu_{\SS'}(E-S,E-T)=\mu_{\SS}(T,S)$. From this perspective, (2) is also immediate.
\end{proof}

What is the point of going through this? Our immediate goal is to determine the conditions under which the expected codimension can be computed with respect to some set other than $\mathcal{P}(E)$ and still give the same answer. To figure this out, it would be enough to establish a condition for when $\ec_{\SS}(M)=\ec_{\SS-\{Z\}}(M)$ for some set $Z$. In fact, we can do a little better:

\begin{prop}
\label{prop:removechange}
If $\SS\subseteq\mathcal{P}(E)$ and $Z\in\SS$, then
\begin{enumerate}
\item $\ec_{\SS}(M)-\ec_{\SS-\{Z\}}(M)=a_{\SS}(Z)b_{\SS}(Z)$.
\item For $S\in\SS-\{Z\}$, $a_{\SS}(S)-a_{\SS-\{Z\}}(S)=a_{\SS}(Z)\mu_{\SS}(Z,S)$.
\item For $S\in\SS-\{Z\}$, $b_{\SS}(S)-b_{\SS-\{Z\}}(S)=\mu_{\SS}(S,Z)b_{\SS}(Z)$.
\end{enumerate}
\end{prop}
\begin{proof}
We have \[\ec_{\SS}(M)-\ec_{\SS-\{Z\}}(M)=\sum_{S,T\in\SS}c(T)(k-\rk S)(\mu_{\SS}(T,S)-\mu_{\SS-\{Z\}}(T,S))\]
if we take $\mu_{\SS-\{Z\}}(T,S)$ to be zero if either $T$ or $S$ is equal to $Z$. Recall that the M\"obius function can be defined by setting $\mu_{\SS}(X,Y)$ to be the sum over all chains in $\SS$ connecting $X$ to $Y$ of $(-1)^c$ where $c$ is the length of the chain. So $\mu_{\SS}(T,S)-\mu_{\SS-\{Z\}}(T,S)$ is going to be the alternating sum of lengths of chains in $\SS$ connecting $T$ to $S$ through $Z$; all other chains will appear in both sums and therefore cancel.

Write $q_k(X,Y)$ for the number of length-$k$ chains in $\SS$ connecting $X$ to $Y$. The number of length-$c$ chains connecting $T$ to $S$ through $Z$ is clearly equal to \[\sum_{k=0}^cq_k(T,Z)q_{c-k}(Z,S),\] which means that \[\mu_{\SS}(T,S)-\mu_{\SS-\{Z\}}(T,S)=\sum_c\sum_k(-1)^k(-1)^{c-k}q_k(T,Z)q_{c-k}(Z,S),\] which is just $\mu_{\SS}(T,Z)\mu_{\SS}(Z,S)$.

Therefore our difference of expected codimensions works out to be \[\sum_{S,T\in\SS}c(T)(k-\rk S)\mu_{\SS}(T,Z)\mu_{\SS}(Z,S)=a_{\SS}(Z)b_{\SS}(Z).\]

Dropping in this expression for the difference of M\"obius functions into the earlier expression of $a$, we see that \[a_{\SS}(S)-a_{\SS-\{Z\}}(S)=\sum_Tc(T)\mu_{\SS}(T,Z)\mu_{\SS}(Z,S)=a_{\SS}(Z)\mu_{\SS}(Z,S),\] and again similarly for $b$.
\end{proof}

\begin{cor}
\label{cor:removemnz}
Given $\AA\subseteq\SS$, if $a_{\SS}(A)=0$ for each $A\in\AA$, then $\ec_{\SS-\AA}(M)=\ec_{\SS}(M)$, and similarly with $a$ replaced with $b$.
\end{cor}
\begin{proof}
Remove the elements of $\AA$ from $\SS$ one by one. By part 1 of the proposition, removing something for which $a=0$ doesn't change $\ec$, and by part 2, the remaining elements of $\AA$ will still have $a=0$ after some have been removed. The argument is exactly analogous for $b$.
\end{proof}

This result will be a lot more useful if we can find a lot of sets for which $a$ and $b$ are zero. Luckily, we can:

\begin{prop}
\label{prop:disconnected}
Suppose that $S\in\SS$ is disconnected, say $S=\bigoplus_iS_i$. Suppose further that, for each $T\subseteq S$ for which $T\in\SS$, we also have that each connected component of $T$ is in $\SS$. Then $a_\SS(S)=0$.
\end{prop}
\begin{proof}
Recall that \[a_\SS(S)=c(S)-\sum_{T\subsetneq S}a_{\SS}(T).\] Any $T\subsetneq S$ which intersects more than one of the $S_i$'s is disconnected, so for those sets we may inductively conclude that $a_\SS(T)=0$. We are left only with sets that are contained in one of the $S_i$'s. To handle those, we note that $\sum_{T\subseteq S_i}a_{\SS}(T)=c(S_i)$. So we are left with $a_\SS(S)=c(S)-\sum_ic(S_i)$. It is simple to check that $c$ is additive in direct sums, so this zero.
\end{proof}

Simply by dualizing everything, we get a version of this statement about $b$. Suppose that $S\in\SS$ is such that $M/S$ is disconnected, and that whenever $T\supseteq S$ is in $\SS$, say $M/T=\bigoplus A_i$, we have each $T\cup A_i\in\SS$. Then $b_\SS(S)=0$.

In particular, Proposition \ref{prop:disconnected} and Corollary \ref{cor:removemnz} together imply that, starting with all of $\mathcal{P}(E)$, we can remove any number of disconnected sets, or any number of sets $S$ for which $M/S$ is disconnected, and end up with the same expected codimension, because the extra condition in Proposition \ref{prop:disconnected} will be trivially satisfied. Note that it doesn't say that we can remove sets of both kinds at the same time: Proposition \ref{prop:removechange} tells us that removing sets for which $b=0$ doesn't change values of $b$ for other sets, but values of $a$ can and will change.

First we need a lemma:

\begin{lem}
\label{lem:connectedness}
Suppose that $M$ is connected and that $S\subseteq M$ is connected. Say $M/S=\bigoplus_iA_i$ where each $A_i$ is connected in $M/S$. Then each $A_i\cup S$ is connected in $M$.
\end{lem}
\begin{proof}
Suppose $A_i\cup S$ is disconnected. Write $A=A_i\cup S$ and $B=\bigcup_{j\ne i}A_j\cup S$, so that $M/S=(A-S)\oplus(B-S)$. Because $S$ is connected, it must be contained in a connected component of $A$. Dually, since $A/S\cong A_i$ is connected, $S$ must contain all but one connected component of $A$. So in fact $S$ is a connected component of $A$, say $A=S\oplus C$.

We know that $\rk M-\rk S=(\rk A-\rk S)+(\rk B-\rk S)$, but our decomposition of $A$ gives us that the right-hand side is $\rk C+\rk B-\rk S$, so $\rk M=\rk B+\rk C$. So in fact $M=B\oplus C$, contradicting the connectedness of $M$.
\end{proof}

Note that by applying the theorem inductively to the $A_i$'s themselves, we get that any $S\cup\bigcup_{i\in I}A_i$ is also connected. Again we can extract a dual version of this statement: if $M$ and $M/S$ are connected but $S=\bigoplus B_i$ with each $B_i$ connected, then each $M/B_i$ is connected.

\begin{thm}
\label{thm:ecinvariance}
Suppose that $M$ is connected, that $\SS$ contains every set $S$ for which both $S$ and $M/S$ are connected, and that whenever $S\in\SS$, all of the connected components of $S$ are also in $\SS$. Then $\ec_{\SS}(M)=\ec(M)$.
\end{thm}
\begin{proof}
Starting with all of $\mathcal{P}(E)$, using Corollary \ref{cor:removemnz} we may remove every set $S$ for which $M/S$ is disconnected and $S\notin\SS$. Call the resulting collection $\TT$. If $T\in\TT-\SS$, we know that $M/T$ is connected, or we would have removed it already. So $T$ must be disconnected, or else it would be in $\SS$. Write $T=\bigoplus_i T_i$ with each $T_i$ connected.

We know that the $T_i$'s themselves are in $\TT$: each $M/T_i$ is connected by the dual version of Lemma \ref{lem:connectedness}, so in fact each $T_i\in\SS$. Suppose $U\subsetneq T$ and $U\in\TT$. If $U\in\TT-\SS$, then $a_\TT(U)=0$ by induction. Otherwise, $U\in\SS$, so its connected components are in $\SS$ by hypothesis, and we again can conclude inductively that $a_\TT(U)=0$. This is enough to be able to apply Proposition \ref{prop:disconnected} to get that $a_\TT(T)=0$.

So by applying Corollary \ref{cor:removemnz} once more, we may remove every set in $\TT-\SS$, which gives the result.
\end{proof}

Note that, in particular, Lemma \ref{lem:connectedness} implies that taking $\SS$ to be the collection of \emph{all} sets $S$ for which both $S$ and $M/S$ are connected will satisfy the hypotheses of Theorem \ref{thm:ecinvariance}. (These sets are called \defof{flacets}, and come up in the study of matroid polytopes. See \cite[2.6]{matpoly}.)

Expected codimension turns out to be well-behaved under direct sums:

\begin{prop}
\label{prop:ecdirectsum}
Let $M$ and $N$ be matroids on sets $E$ and $F$, and take collections $\SS\subseteq\mathcal{P}(E)$ and $\TT\subseteq\mathcal{P}(F)$. In $\mathcal{P}(E\cup F)$, let \[\AA=\SS\cup\TT\cup\{S\cup T:S\in\SS,\ T\in\TT\}.\] Then $\ec_\AA(M\oplus N)=\ec_\SS(M)+\ec_\TT(N)$.
\end{prop}
\begin{proof}
Take $A\in\AA$. If $A$ is a union of nonempty sets from $\SS$ and $\TT$, then it satisfies the hypotheses of Proposition \ref{prop:disconnected}. Otherwise, if $A\in\SS$, then $a_\AA(A)=a_\SS(A)$, and similarly for $\TT$. So in fact, \begin{align*}\ec_\AA(M\oplus N)&=\sum_{A\in\AA}a_\AA(A)\codim(A)\\&=\sum_{A\in\SS}a_\SS(A)\codim(A)+\sum_{A\in\TT}a_\TT(A)\codim(A)\\&=\ec_\SS(M)+\ec_\TT(N).\end{align*}
\end{proof}

In particular, using Proposition \ref{prop:directsum}, we see that if $M$ and $N$ have expected codimension, so does $M\oplus N$. Since it is trivial to check that both matroids on a one-element set have expected codimension, this also applies to loop and coloop extensions.

\begin{figure}[t]
\begin{center}\includegraphics[width=8cm]{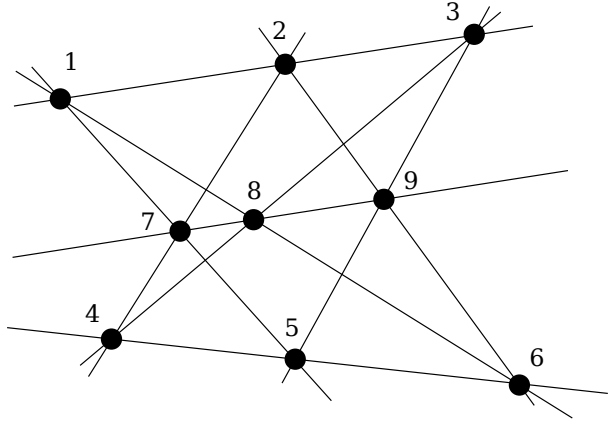}\end{center}
\caption{A projective model of the ``Pappus matroid.''}
\label{fig:pappus}
\end{figure}

We conclude this section with an example of a matroid that doesn't have expected codimension:

\begin{cex}
\label{cex:pappus}
Consider the \defof{Pappus matroid} $P$, the rank-3 matroid on $[9]$ generated by the collinearities in Figure \ref{fig:pappus}. The only sets $S\subseteq[9]$ for which both $S$ and $P/S$ are connected are the nine sets of points which lie on lines in the picture. (That is, 123, 456, 789, 157, 168, 247, 269, 348, and 359.) From this we can easily compute that $\ec(P)=9$. However, the actual codimension of $X(P)$ in $G(3,9)$ is 8. This can be (and was) computed directly with a computer algebra system like Macaulay2; it also follows from computations performed in \cite{fnrmat}. Either way, $P$ doesn't have expected codimension.

This should not be especially surprising: the whole point of the Pappus matroid is that it demonstrates Pappus's theorem, that is, the fact that given any eight of the collinearities in Figure \ref{fig:pappus}, the ninth comes for free. Our definition of expected codimension is unable to keep track of ``global'' constraints like this one, so it treats all nine rank conditions as independent.
\end{cex}

\section{Positroids}
We have already seen in Counterexample \ref{cex:pappus} a case in which the expected codimenion of a matroid variety fails to line up with its actual codimension in the Grassmannian. The main result of this section is that that doesn't happen for positroids:

\begin{thm}
\label{thm:posexpcodim}
Positroids have expected codimension.
\end{thm}

In order to prove this, we're going to need to understand a little bit more about the matroid structure of a positroid. We refer repeatedly to restrictions and contractions of positroids; note that it follows directly from the second definition in Proposition \ref{prop:posdefs} that these are again positroids.

\begin{lem}
\label{lem:connectedintervals}
If $P$ is a positroid on $[n]$, $X\subseteq [n]$, and both $P|_X$ and $P/X$ are connected, then $X$ is an interval.
\end{lem}
\begin{proof}
Suppose $X$ is not an interval. Take $c_1,c_2\in[n]-X$ to lie in two different cyclic intervals of $[n]-X$. Since $P/X$ is connected, there is a circuit $C$ of $P/X$ which contains both $c_1$ and $c_2$. By restricting to $X\cup C$, we may assume that $P/X$ is a circuit. Similarly, for $b_1,b_2\in X$ lying on different sides of $c_1$ and $c_2$ (so the named elements appear in the cyclic order $b_1,c_1,b_2,c_2$), there is a circuit $B$ of $(P|_X)^*=P/([n]-X)$ containing both, and we may contract the elements of $X-B$ and assume that $(P|_X)^*$ is a circuit, that is, that everything in $X$ is parallel.

Now, delete all elements of $X$ other than $b_1$ and $b_2$. This doesn't change the rank of any set in $P/X$: everything in $X$ was parallel to $b_1$ and $b_2$. Dually, contract all elements of $[n]-X$ except $c_1$ and $c_2$. Now we have $n=4$, and the sets $\{1,3\}$ and $\{2,4\}$ each have rank 1. This matroid is not a positroid, which can easily be checked, so we have a contradiction.
\end{proof}

\begin{lem}
\label{lem:noncrossing}
The connected components of a positroid form a non-crossing partition. (This was also proved independently in \cite{noncross}.)
\end{lem}
\begin{proof}
Suppose first that $P$ has just two connected components, say $P=A\oplus B$. Then $P/A=B$ and $P/B=A$ are also both connected, so Lemma \ref{lem:connectedintervals} implies that they are both intervals. If there are more than two connected components, they no longer have to both be intervals, but for any two components $C$ and $D$, each of $C$ and $D$ must be an interval inside $C\cup D$, which means in particular that they cannot cross.
\end{proof}

\begin{lem}
\label{lem:intervalcomponents}
If $P$ is a connected positroid on $[n]$ and $I\subseteq [n]$ is an interval, then each connected component of $I$ is an interval.
\end{lem}
\begin{proof}
Say $I=X\oplus\bigoplus_iY_i$ with $X$ and each $Y_i$ connected, and suppose $X$ is not an interval, say $X=\bigcup_kJ_k$ and $I-X=\bigcup_lJ'_l$ where each $J_k$ and $J'_l$ is an interval. Since the components of $I$ have to form a non-crossing partition by Lemma \ref{lem:noncrossing}, none of the $Y_i$'s can meet more than one of the $J'_l$'s. So we may assume that left and right endpoints of $X$ coincide with those of $I$ by removing the $Y_i$'s that lie to the left of $X$'s left endpoint or to the right of its right endpoint. We now know that all the $J'_l$'s lie in between two $J_k$'s.

Just as in the proof of Lemma \ref{lem:connectedintervals}, the connectedness of $X$ lets us conclude that there is a circuit of $X^*=P/([n]-X)$ that contains points in two different $J_k$'s. Suppose there is a circuit of $P/X$ that contains a point of $I-X$ and a point of $P-I$. If this were the case, because we forced all the points of $I-X$ to lie between intervals of $X$, we would be in the exact situation that gave us a contradiction in the previous proof.

So there must be no such circuit. But this means that \[P/X=I/X\oplus((P-I)\cup X)/X=\left(\bigoplus_iY_i\right)\oplus(P-I\cup X/X),\] which means that \[\rk P-\rk X=\sum_i\rk Y_i+\rk(P-I\cup X)-\rk X,\] so in fact \[P=\left(\bigoplus_iY_i\right)\oplus((P-I)\cup X),\] contradicting the connectedness of $P$.
\end{proof}

\begin{lem}
\label{lem:npositroid}
For a positroid $P$, let $\II\subseteq\mathcal{P}([n])$ be the collection of all cyclic intervals. For any interval $[i,j]\ne[n]$, $a_\II([i,j])$ is equal to the entry (either 0 or 1) at $(i,j)$ in the affine permutation matrix.
\end{lem}
\begin{proof}
To see this, it's enough to note that our purported $a_\II$ satisfies the relation \[c(S)=\sum_{T\subseteq S}a_\II(T).\] We mentioned above that $c([i,j])=\#[i,j]-\rk[i,j]$ is the number of intervals $[k,l]\subseteq[i,j]$ with a 1 in the affine permutation matrix at position $(k,l)$, so this is true.
\end{proof}

\begin{proof}[Proof of Theorem \ref{thm:posexpcodim}]
First, we claim that for a positroid $P$, if $\II\subseteq\mathcal{P}([n])$ is the collection of all cyclic intervals of $[n]$, $\ec(P)=\ec_\II(P)$. For connected positroids, this follows immediately from Theorem \ref{thm:ecinvariance}: Lemmas \ref{lem:connectedintervals} and \ref{lem:intervalcomponents} say that taking $\SS=\II$ satisfies the hypotheses of the theorem. For a general positroid, we can decompose it as a direct sum and apply Proposition \ref{prop:ecdirectsum}.

So now it remains to show that $\ec_\II(P)$ is the actual codimension of $P$. We'll use the fact that $\codim P=l(\pi)$, where $\pi$ is the corresponding affine permutation. Let $d_I$ be the number of intervals $[k,l]\subseteq I$ with a 1 in position $(k,l)$, and let $\II'$ be the collection of intervals other than $[n]$ with a 1 at the corresponding position. We can ignore $[n]$ itself because $\codim[n]=0$, so we may compute, using Lemma \ref{lem:npositroid}:
\[\ec_\II(P)=\sum_{I\in\II}a_\II(I)\codim(I)=\sum_{I\in\II'}\codim(I)=\sum_{I\in\II'}(k-(\#I-d_I)),\] where $k$ is the rank of $P$. A simple computation (which is spelled out in \cite{klspos}) shows that $\sum\#I=nk+n$, and we know that $\sum d_I=l(\pi)+n$, since we're counting the pairs of intervals in the definition of $l(\pi)$ and also counting the pairs $(I,I)$. So we're left with \[\ec_\II(P)=nk-(nk+n)+l(\pi)+n=l(\pi).\]
\end{proof}

\section{Valuativity}
\label{sec:valuativity}
Let $M$ be a matroid on a set $[n]$. For each basis $B$ of $M$, consider the vectors in $\R^n$ whose entries are 1 if the corresponding element of $[n]$ is in $B$ and 0 otherwise. The convex hull of these vectors is called the \defof{matroid polytope} of $M$, written $P(M)$. There are many examples of combinatorial properties of matroids that are encoded in the geometry of the matroid polytope.

\begin{defn}
Given a matroid $M$, a \defof{matroidal subdivision} of $P(M)$ is a decomposition of $P(M)$ into polytopes which are all matroid polytopes. If $\mathcal{D}$ is a matroidal subdivision of $P(M)$, write $\mathcal{D}_{in}$ for the internal faces of $\mathcal{D}$, that is, the faces of $\mathcal{D}$ that are not also faces of $P(M)$. Given a function $f$ from the set $\mathrm{Mat}(n)$ of all matroids on $[n]$ into an abelian group, we say $f$ is \defof{valuative} if, for any matroid $M$ and any matroidal subdivision $\mathcal{D}$ of $P(M)$, \[f(M)=\sum_{P(N)\in\mathcal{D}_{in}}(-1)^{\dim P(M)-\dim P(N)}f(N).\]
\end{defn}

Valuative matroid invariants are studied in detail in \cite{valuative}. We single out the following result, which appears as \cite[5.4]{valuative} in slightly different language:

\begin{thm}
\label{thm:valschub}
The set of Schubert matroids forms a basis for $\mathrm{Mat}(n)$ modulo matroidal subdivisions.
\end{thm}

We will show:

\begin{thm}
\label{thm:expval}
Expected codimension is a valuative matroid invariant.
\end{thm}

Since Schubert matroids, being positroids, have expected codimension, Theorem \ref{thm:valschub} gives us another way to think about expected codimension: you could have defined it by assigning Schubert matroids their codimensions and extending to all matroids by subdividing the matroid polytope and insisting that it be valuative.

We'll prove Theorem \ref{thm:expval} by proving something stronger first:

\begin{lem}
\label{lem:tuttelike}
Let $M$ be a matroid on $[n]$, and define \[s_M(x,y,z)=\sum_{S\subseteq T\subseteq [n]}x^{\#S-\rk S}y^{\rk M-\rk T}z^{\#T-\#S}.\] Then the function that takes $M$ to $s_M$ is valuative.
\end{lem}

This is a generalization of the \defof{Tutte polynomial}, which is \[t_M(x,y)=s_M(x-1,y-1,0).\] In \cite[6.4]{troplin}, David Speyer shows that the Tutte polynomial is a valuative matroid invariant. The proof we give here of \ref{lem:tuttelike} turns out to be almost identical. We single out the following lemma, which appears as \cite[6.5]{troplin}:

\begin{lem}
\label{lem:euler}
If $P$ is a polytope and $\mathcal{G}$ is the set of internal faces of a decomposition of $P$, then \[\sum_{N\in\mathcal{G}}(-1)^{\dim P-\dim N}=1.\]
\end{lem}
\begin{proof}
This is just $(-1)^{\dim P}(\chi(P)-\chi(\partial P))$ where $\chi$ is the Euler characteristic. So the result follows, becuase $P$ is contractible and $\partial P$ is homeomorphic to a $(\dim P-1)$-sphere.
\end{proof}

\begin{proof}[Proof of Lemma \ref{lem:tuttelike}]
Plugging the definition of $s_M$ to the definition of valuativity, it's enough to show, for any matroidal subdivision $\mathcal{D}$ of $P(M)$, \[x^{\#S-\rk_M S}y^{\rk M-\rk_M T}z^{\#T-\#S}\] is equal to \[\sum_{P(N)\in\mathcal{D}_{in}}(-1)^{\dim P(M)-\dim P(N)}x^{\#S-\rk_N S}y^{\rk N-\rk_N T}z^{\#T-\#S}.\] Comparing the coefficients of $x^{\#S-\rk_M S}y^{\rk M-\rk_M T}z^{\#T-\#S}$, we want that \[\sum_{\substack{P(N)\in\mathcal{D}_{in}\\\rk_N(S)=r,\ \rk_N(T)=s}}(-1)^{\dim P(M)-\dim P(N)}\] is 1 if $(r,s)=(\rk_M(S),\rk_M(T))$ and 0 otherwise. Since the sum is empty if $r>\rk_M(S)$ or $s>\rk_M(T)$, we'll just show that \[\sum_{\substack{P(N)\in\mathcal{D}_{in}\\\rk_N(S)\ge r,\ \rk_N(T)\ge s}}(-1)^{\dim P(M)-\dim P(N)}=1\] for $r\le \rk_M(S)$ and $s\le \rk_M(T)$.

Let $l_S$ be the linear function on $\R^n$ sending $(x_i)$ to $\sum_{i\in S}x_i$. Note that \[\rk_N(S)=\max_{x\in P(N)}l_S(x).\] So $\rk_N(S)\ge r$ if and only if $P(N)$ intersects the half-space $l_S>r-\frac12$. So our equality follows by applying Lemma \ref{lem:euler} to $P(M)\cap\{l_M>r-\frac12\}\cap\{l_N>s-\frac12\}$.
\end{proof}

\begin{proof}[Proof of Theorem \ref{thm:expval}]
This follows directly: \[\ec(M)=\frac\partial{\partial x}\frac\partial{\partial y}s_M(0,0,-1).\]
\end{proof}

\chapter{Code}
In the course of doing the research described here, I found it helpful to write a script in Macaulay2 which I've uploaded to my Web site at \texttt{http://nicf.net/static/posshift.m2}. This section is a short description of some of the functions defined in that file and how to use them. The descriptions of the functions below are followed by a short example.

\begin{itemize}
\item \texttt{initialize}: This function should be called before doing anything else. It takes three parameters, \texttt{n}, \texttt{k}, and \texttt{stief}, and it defines the rings necessary to work in the Grassmannian $G(k,n)$. If \texttt{stief} is true, then computations are done instead in the Stiefel cone, that is, the variety of full-rank $k\times n$ matrices.
\item \texttt{ssToMatroid}: This function takes a list of numbers which define the siteswap of a juggling pattern (as described in \cite{klspos}) and produces a matroid. Matroids are stored as lists of lists of numbers, each of which is a basis.
\item \texttt{matroidToSS}: This is the partial inverse to the previous function.
\item \texttt{matroidToIdeal}: This takes a matroid and returns the ideal of the subvariety of the Grassmannian defined by the vanishing of the Pl\"ucker coordinates corresponding to nonbases of the matroid. (Note that this is \emph{not} always the matroid variety.)
\item \texttt{matroidToIdeal'}: This function is like the previous one, but returns the ideal of the matroid variety. It's often much slower.
\item \texttt{idealToMatroid}: This is the partial inverse to \texttt{matroidToIdeal}. (Note: not to \texttt{matroidToIdeal'}.)
\item \texttt{ssToIdeal}: This is the composition of \texttt{matroidToIdeal} and \texttt{ssToMatroid}.
\item \texttt{idealToSS}: This is the composition of \texttt{matroidToSS} and \texttt{idealToMatroid}.
\item \texttt{geomShift}: This function performs a geometric shift. It takes three parameters, \texttt{i}, \texttt{j}, and \texttt{I}, and outputs the result of the shift $\Sha_{j\to i}I$.
\end{itemize}

\subsection*{Example}
This example shows the result of shifting 2 to 5 in the subvariety of $G(3,5)$ defined by the ideal $(p_{123},p_{345})$. We end up with two components: one where 5 is a coloop and one where 2 and 3 are parallel.

\begin{lstlisting}
i1 : initialize(5,3,false)

o1 = false

i2 : ssToIdeal {2,4,2,3,4}

o2 = ideal (p         , p         )
             {1, 2, 3}   {3, 4, 5}

o2 : Ideal of G

i3 : geomShift(5,2,o2)

o3 = ideal (p         , p         )
             {2, 3, 4}   {1, 2, 3}

o3 : Ideal of G

i4 : decompose o3

o4 = {ideal (p         , p         , p         , p         ), ideal (p         ,
              {2, 3, 4}   {1, 3, 4}   {1, 2, 4}   {1, 2, 3}           {2, 3, 5} 
     ---------------------------------------------------------------------------
     p         , p         )}
      {2, 3, 4}   {1, 2, 3}

o4 : List

i5 : idealToMatroid \ o4

o5 = {{{1, 2, 5}, {1, 3, 5}, {2, 3, 5}, {1, 4, 5}, {2, 4, 5}, {3, 4, 5}}, {{1,
     ---------------------------------------------------------------------------
     2, 4}, {1, 3, 4}, {1, 2, 5}, {1, 3, 5}, {1, 4, 5}, {2, 4, 5}, {3, 4, 5}}}

o5 : List

\end{lstlisting}

\bibliography{thesisbib}

\end{document}